\documentclass[11pt,reqno]{amsart}

\usepackage[T1]{fontenc}
\usepackage[letterpaper,margin=1in]{geometry}
\usepackage{mathpazo}
\usepackage{microtype}
\usepackage{mathtools}
\usepackage{amssymb}
\usepackage{mathrsfs}
\usepackage{enumitem}
\usepackage{xcolor}
\usepackage[
  colorlinks=true,
  linkcolor=blue!45!black,
  citecolor=green!35!black,
  urlcolor=blue!55!black,
  pdfauthor={Ajneet Dhillon and Brett Nasserden},
  pdftitle={Finite extensions of abelian schemes and duality for commutative group stacks},
  pdfsubject={Duality for proper flat commutative group stacks},
  pdfkeywords={commutative group stack, Picard stack, Cartier duality,
    quasiabelian group scheme, Fourier--Mukai duality, Neron model}
]{hyperref}
\usepackage[nameinlink,noabbrev,capitalise]{cleveref}

\allowdisplaybreaks
\setlist[enumerate,1]{label=\textup{(\roman*)},leftmargin=2.2em}
\setlist[itemize]{leftmargin=2em}

\newtheorem{theorem}{Theorem}[section]
\newtheorem{proposition}[theorem]{Proposition}
\newtheorem{lemma}[theorem]{Lemma}
\newtheorem{corollary}[theorem]{Corollary}

\theoremstyle{definition}
\newtheorem{definition}[theorem]{Definition}
\newtheorem{example}[theorem]{Example}
\theoremstyle{remark}
\newtheorem{remark}[theorem]{Remark}

\newcommand{\cA}{\mathscr A}
\newcommand{\cB}{\mathscr B}
\newcommand{\cC}{\mathscr C}
\newcommand{\cF}{\mathscr F}
\newcommand{\cG}{\mathscr G}
\newcommand{\cK}{\mathscr K}
\newcommand{\sP}{\mathscr P}
\newcommand{\cX}{\mathscr X}

\newcommand{\cE}{{\mathcal E}}
\newcommand{\cH}{{\mathcal H}}
\newcommand{\cL}{{\mathcal L}}
\newcommand{\cM}{{\mathcal M}}
\newcommand{\cN}{{\mathcal N}}
\newcommand{\cO}{{\mathcal O}}
\newcommand{\cP}{{\mathcal P}}
\newcommand{\cQ}{{\mathcal Q}}
\newcommand{\cV}{{\mathcal V}}
\newcommand{\Gm}{\mathbf G_{\mathrm m}}
\newcommand{\BG}[1]{\mathrm B#1}
\newcommand{\HHm}{\mathrm H^{-1}}
\newcommand{\HHz}{\mathrm H^{0}}
\newcommand{\Dual}[1]{\mathbf D(#1)}
\newcommand{\DoubleDual}[1]{\mathbf D\mathbf D(#1)}
\newcommand{\CartierDual}[1]{#1^{\mathrm D}}
\newcommand{\AbDual}[1]{#1^{\vee}}
\newcommand{\ExtSheaf}[2]{\mathop{\cE\!xt}\nolimits^{#1}(#2,\Gm)}
\newcommand{\RHom}{\mathop{\mathbf R\!\operatorname{Hom}}}
\newcommand{\HomStack}{\mathop{\cH\!om}\nolimits}
\newcommand{\Pic}{\operatorname{Pic}}

\newcommand{\Spec}{\operatorname{Spec}}
\newcommand{\Ker}{\operatorname{Ker}}
\newcommand{\Coker}{\operatorname{Coker}}

\newcommand{\id}{\operatorname{id}}
\newcommand{\fppf}{\mathrm{fppf}}
\newcommand{\etale}{\mathrm{\acute et}}
\newcommand{\generic}{\eta}

\newcommand{\Frac}{\operatorname{Frac}}
\newcommand{\Neron}{N\'{e}ron}
\newcommand{\To}{\longrightarrow}

\newcommand{\DQcoh}[1]{\mathbf D\bigl(\operatorname{Qcoh}(#1)\bigr)}
\newcommand{\DbCoh}[1]{\mathbf D^{\mathrm b}\bigl(\operatorname{Coh}(#1)\bigr)}

\DeclareMathOperator{\Isom}{Isom}

\DeclareMathOperator{\NS}{NS}

\title[Finite extensions of abelian schemes and duality]
  {Finite extensions of abelian schemes and duality\\
   for commutative group stacks}

\author{Ajneet Dhillon}
\address{Department of Mathematics, Western University,
London, Ontario, Canada}

\author{Brett Nasserden}
\address{Department of Mathematics and Statistics, McMaster University,
Hamilton, Ontario, Canada}

\date{August 25, 2026}

\begin{document}

\begin{abstract}
Let $R$ be a discrete valuation ring.  We prove that every proper flat
finitely presented commutative $R$-group scheme $P$ fits into an exact
sequence
\[
  0\longrightarrow E\longrightarrow P\longrightarrow B\longrightarrow0,
\]
where $E$ is finite flat and $B$ is an abelian scheme.  For a quasiabelian
model---that is, a proper flat finitely presented $R$-group scheme with
abelian generic fibre---we prove a finer structure theorem: its normalization
is an abelian scheme, and the model is obtained from it by a pushout involving
finite flat group schemes.  An analogous abelian quotient exists for proper
flat commutative group stacks with finite flat inertia.  

These structure theorems give explicit quotient presentations for duals.
  Over an arbitrary base, we
show that the dual of a proper flat finitely presented commutative group
algebraic space is algebraic, proper, flat, and finitely presented, and
biduality holds.  When $2$ is invertible, the corresponding result for group
stacks proves Brochard's duality conjecture.

As an application over a discrete valuation ring, Polishchuk's kernel-algebra
Fourier--Mukai argument gives equivalences between the unbounded derived
categories of quasi-coherent sheaves on every proper flat finitely presented
commutative group scheme and its dual, and between their bounded derived
categories of coherent sheaves, with no tameness or residue-characteristic
hypothesis.  For a quasiabelian model, the dual category is a finite-flat
equivariant category on its abelian normalization.
\end{abstract}

\maketitle

\section{Introduction}

Let $S$ be a scheme and let $\cG$ be a strictly commutative group stack over
$S$.  Following Deligne and Brochard, its dual is
\[
  \Dual{\cG}
  :=\HomStack(\cG,\BG{\Gm}).
\]
There is a canonical evaluation homomorphism
\[
  e_{\cG}\colon \cG\longrightarrow \DoubleDual{\cG}.
\]
Brochard conjectured that, when $2$ is invertible on $S$, the dual of a proper
flat finitely presented commutative group stack with finite flat inertia is
again proper and flat, and that every such stack is dualizable
\cite[Conjecture~1.1]{BrochardDuality}.    More precisely, the conjecture
is the following.

\begin{theorem}[Brochard's conjecture]\label{thm:intro-brochard}
Let $S$ be a scheme on which $2$ is invertible, and let $\cG$ be a proper,
flat, finitely presented algebraic commutative group stack over $S$ whose
inertia is finite and flat over $\cG$.  Then:
\begin{enumerate}
  \item $\Dual{\cG}$ is algebraic, proper, flat, and finitely presented over
  $S$, and its diagonal is finite;
  \item the evaluation morphism
  $e_{\cG}\colon\cG\To\DoubleDual{\cG}$ is an isomorphism.
\end{enumerate}
\end{theorem}

Brochard proved algebraicity and finite presentation of $\Dual{\cG}$, the
affineness of its diagonal, and properness of its geometric fibres
\cite[Theorem~3.14]{BrochardDuality}.  He also showed that the second assertion
of \cref{thm:intro-brochard} follows from the first
\cite[Remark~4.14]{BrochardDuality}.  The remaining issues are flatness and
global properness of the dual; once properness is known, its quasi-finite
affine diagonal becomes finite.  

The point is subtle because a proper flat group scheme whose generic fibre is
abelian need not be an abelian scheme.  Serre's mixed-characteristic example,
and its equal-characteristic analogue, arise as schematic images of
homomorphisms of abelian schemes and are proper and flat but nonsmooth
\cite[Examples~4.11 and~4.12]{AchterCasalainaMartinWise}.  In particular, one
cannot replace every such model by an abelian scheme or assume cohomological
flatness in degree zero.

Our first result gives a direct  replacement.  We say that $P$ is a
\emph{finite extension of an abelian scheme} if it occurs in an exact sequence
as below. 

\begin{theorem}\label{thm:intro-group-structure}
Let $R$ be a discrete valuation ring and let $P$ be a proper flat finitely
presented commutative $R$-group scheme.  There is an exact sequence
\[
  0\longrightarrow E\longrightarrow P\longrightarrow B\longrightarrow0,
\]
where $E$ is finite flat and $B$ is an abelian scheme.
\end{theorem}

No assumption on $2$ is needed.  The proof
starts with the closure $N$ of the maximal abelian subgroup of the generic
fibre.  A polarization produces a finite faithfully flat quotient
$N\to B$.  If $Q=P/N$ and $m=\operatorname{rk}(Q)$, then $[m]_P$ factors
through $N$; composing the resulting map $P\to N$ with $N\to B$ gives the
asserted quotient in one step.

For a quasiabelian model the construction is considerably more explicit.

\begin{theorem}\label{thm:intro-structure}
Let $R$ be a discrete valuation ring and let $N$ be a proper flat finitely
presented group scheme whose generic fibre is an abelian variety.  Then $N$ is
commutative.  The \Neron{} model $A$ of its generic fibre is an abelian scheme,
and the extension of the generic identity
\[
  u\colon A\longrightarrow N
\]
is finite and identifies $A$ with the normalization of $N$.

For every $R$-ample invertible sheaf $\cL$ on $N$, put
$\cM=u^*\cL$, $K=\Ker(\phi_{\cM})$, and
$F=\Ker(q_\cL)$, where
\[
  q_\cL\colon N\longrightarrow A^\vee,
  \qquad q_\cL u=\phi_{\cM},
\]
is the polarization quotient.  Then $F$ is finite flat and, for the induced
map $h:K\to F$, there is a bicartesian diagram
\[
\begin{array}{ccccccccc}
0&\to&K&\to&A&\xrightarrow{\phi_{\cM}}&A^\vee&\to&0\\
 &&\downarrow h&&\downarrow u&&\Vert\\
0&\to&F&\to&N&\xrightarrow{q_\cL}&A^\vee&\to&0.
\end{array}
\]
Equivalently, $N\simeq(A\times_R F)/K$, where
$k\mapsto(k,-h(k))$, and $h$ is an isomorphism on the generic fibre.
\end{theorem}

On a geometric special fibre, put $G=N_{\bar s}$.  There is an exact
sequence
\[
  0\longrightarrow G_{\mathrm{red}}\longrightarrow G
  \longrightarrow C\longrightarrow0,
\]
where $G_{\mathrm{red}}$ is an abelian variety and $C$ is finite and
connected.  The specialization of $u$ induces an isogeny
$v:A_{\bar s}\to G_{\mathrm{red}}$, and
\[
  e(N):=\deg(v)=\operatorname{rk}(C).
\]
If the residue characteristic is $p>0$, then $e(N)$ is a power of $p$; in
residue characteristic zero, $e(N)=1$.  Writing
$g=\dim G_{\mathrm{red}}$, the fundamental cycle of $G$ and the coherent
cohomology of its structure sheaf satisfy
\[
  [G]=e(N)[G_{\mathrm{red}}],\qquad
  \dim_{\bar k}H^i(G,\cO_G)
    =e(N)\binom gi\quad(i\geq0).
\]
The automorphism group scheme of the neutral object of $\Dual G$ is
$G^{\mathrm D}\simeq C^{\mathrm D}$, and hence has rank $e(N)$.  Finally,
for every $R$-ample invertible sheaf $\cL$ on $N$,
\[
  e(N)^2\mid\deg\bigl(\phi_{u^*\cL}\bigr).
\]
Thus, in residue characteristic $p>0$, if this polarization has degree prime
to $p$, then $e(N)=1$ and the normalization morphism $u:A\to N$ is an
isomorphism.

The dual statements become shorter and stronger when one uses
\cite[Lemma~3.9]{BrochardDuality}.  From an exact sequence
$0\to E\to P\to B\to0$ one obtains, in every characteristic,
\[
  \Dual P\simeq[B^\vee/E^{\mathrm D}].
\]
The finite faithfully flat cover $B^\vee\to[B^\vee/E^{\mathrm D}]$ proves
properness and flatness directly.  Globalizing gives the following theorem.

\begin{theorem}\label{thm:intro-group-duality}
Let $S$ be a scheme and let $P$ be a proper flat finitely presented
commutative group algebraic space over $S$.  Then $\Dual P$ is algebraic,
proper, flat, and finitely presented over $S$, with finite diagonal.  The
evaluation map $P\to\DoubleDual P$ is an isomorphism.  No assumption on $2$ is
required.
\end{theorem}

The condition on the invertibility of $2$ arises from the passage from group
algebraic spaces to group stacks.  For a group stack $\cG$, the canonical
sequence
\[
  0\longrightarrow\BG{\HHm(\cG)}\longrightarrow\cG
  \longrightarrow\HHz(\cG)\longrightarrow0
\]
contributes the obstruction
$\CartierDual{\HHm(\cG)}\to
\ExtSheaf{2}{\HHz(\cG)}$.  Ribeiro and Rosengarten prove that the second Ext
sheaf of an abelian scheme vanishes over an arbitrary base
\cite[Theorem~4.3]{RibeiroRosengarten}.  In the finite-by-abelian
d\'{e}vissages used below, the remaining possible contribution embeds in the
second Ext sheaf of the finite flat kernel.  Breen's vanishing theorem kills
this term when $2$ is invertible; this is the source of the hypothesis in
\cref{thm:intro-brochard}.  The special-fibre calculation also explains why
the output is stated with finite, rather than finite flat, diagonal: duality
exchanges vertical nilpotent thickness with vertical stackiness.

The discrete-valuation structure theorem also has derived-categorical and
geometric applications.  For an algebraic stack $X$, write $\DQcoh{X}$ for
the unbounded derived category of the abelian category of quasi-coherent
sheaves, and write $\DbCoh{X}$ for the bounded derived category of coherent
sheaves.

\begin{theorem}\label{thm:intro-fourier-mukai}
Let $R$ be a discrete valuation ring and let $P$ be a proper flat finitely
presented commutative $R$-group algebraic space.  There are exact
Fourier--Mukai equivalences
\[
  \DQcoh{P}\simeq\DQcoh{\Dual P},
  \qquad
  \DbCoh{P}\simeq\DbCoh{\Dual P}.
\]
They require neither that $2$ be invertible nor that the finite kernel in a
finite-by-abelian presentation of $P$, or its Cartier dual, be linearly
reductive.

If $P=N$ is quasiabelian, with normalization $u:A\to N$ and pushout datum
$h:K\to F$ as in \cref{thm:intro-structure}, then the bounded equivalence has
the concrete form
\[
  \DbCoh{N}
  \simeq \DbCoh{[A/F^{\mathrm D}]}
  \simeq
  \mathbf D^{\mathrm b}\bigl(\operatorname{Coh}^{F^{\mathrm D}}(A)\bigr),
\]
where $F^{\mathrm D}$ acts on the abelian normalization $A$ by translation
through $F^{\mathrm D}\xrightarrow{h^{\mathrm D}}K^{\mathrm D}\to A$.
\end{theorem}

The proof uses the explicit finite kernel algebras in Polishchuk's
Fourier--Mukai argument \cite{PolishchukKernel}.  In this form it does not
require a higher-Ext concentration statement for finite flat group schemes,
which is important in residue characteristic $2$.  It also applies to an
abelian stack $\cA$ over $R$: Brochard identifies $G=\Dual{\cA}$ as a
duabelian proper flat group scheme and $\cA\simeq\Dual G$, so
\[
  \DQcoh{\cA}\simeq\DQcoh{G},
  \qquad
  \DbCoh{\cA}\simeq\DbCoh{G}.
\]
For the derived categories of the abelian categories of sheaves, this gives an
analogue without a tameness assumption of the Mukai duality theorem in
\cite{DhillonNasserdenMukai}.  That theorem concerns stable
$\infty$-categories; we do not identify $\DQcoh{\cA}$ with the homotopy
category of the stable $\infty$-category appearing there, nor assert a stable
enhancement here.

The same finite by abelian presentation has a geometric consequence beyond
the quasiabelian case: every $P$ in \cref{thm:intro-fourier-mukai} is
Gorenstein over $R$, its relative dualizing sheaf is noncanonically trivial,
and, if $g=\dim P_{\Frac(R)}$, a choice of trivialization gives
\[
  \RHom_R\bigl(\mathbf R\Gamma(P,\cO_P),R\bigr)
  \simeq \mathbf R\Gamma(P,\cO_P)[g].
\]

The organization is as follows.  In \cref{sec:background} we recall the exact
sequences, Ext statements, and quotient presentation used in the argument.
The normalization, polarization quotient, pushout normal form, and defect are
proved in \cref{sec:quasiabelian}.  Their duality consequences appear in
\cref{sec:local-duality}.  In \cref{sec:devissage} we pass to arbitrary proper
flat group schemes and then to commutative group stacks over a discrete
valuation ring.  The characteristic-free group-space theorem and Brochard's
conjecture are globalized in \cref{sec:global}.  Fourier--Mukai duality and the
general Gorenstein consequence are then recorded in \cref{sec:applications}.
Arithmetic consequences over Dedekind schemes and low-ramification bases are
collected in \cref{sec:arithmetic}, followed by examples in
\cref{sec:examples}.

\section{Commutative group stacks and duality}\label{sec:background}

We work on the big $\fppf$ site of the base.  All group schemes are
commutative unless explicitly stated otherwise.  Exactness for commutative
group stacks is understood in the sense of
\cite[Definition~2.11]{BrochardDuality}.
Throughout, a finite flat group scheme is understood to be finite locally
free.

\subsection{Cohomology sheaves and the dual}

By Deligne's equivalence, a commutative group stack $\cG$ is represented by a
complex $\cG^{\flat}$ of abelian sheaves concentrated in degrees $[-1,0]$;
see \cite[Expos\'{e}~XVIII, \S1.4]{SGA4} and
\cite[Theorem~2.5]{BrochardDuality}.  We write $\HHm(\cG)$ for the
automorphism group of the neutral object and $\HHz(\cG)$ for the sheaf of
isomorphism classes.  There is a canonical exact sequence
\begin{equation}\label{eq:canonical-sequence}
  0\longrightarrow \BG{\HHm(\cG)}\longrightarrow \cG
  \longrightarrow \HHz(\cG)\longrightarrow 0.
\end{equation}

The dual is represented by
\begin{equation}\label{eq:dual-truncation}
  \Dual{\cG}^{\flat}
  \simeq \tau_{\leq 0}\RHom(\cG^{\flat},\Gm[1]).
\end{equation}
Consequently, there is a functorial identification and an exact sequence
\begin{equation}
  \HHm(\Dual{\cG})
    \simeq \CartierDual{\HHz(\cG)}.
  \label{eq:hminus-dual}
\end{equation}
\begin{equation}
  \begin{aligned}
    0\longrightarrow \ExtSheaf{1}{\HHz(\cG)}
      &\longrightarrow \HHz(\Dual{\cG})\\
      &\longrightarrow \CartierDual{\HHm(\cG)}
      \longrightarrow \ExtSheaf{2}{\HHz(\cG)}.
  \end{aligned}
  \label{eq:hzero-dual}
\end{equation}
These are \cite[Lemma~3.4]{BrochardDuality}.

For later use, we record the exactness criterion in the form that we need.

\begin{lemma}\label{lem:dual-exactness}
Let
\[
  0\longrightarrow \cA\longrightarrow \cB\longrightarrow \cC
  \longrightarrow 0
\]
be an exact sequence of commutative group stacks.
\begin{enumerate}
  \item If
  $\ExtSheaf{2}{\HHz(\cC)}=0$ and
  $\ExtSheaf{1}{\HHm(\cC)}=0$, then
  \[
    0\longrightarrow \Dual{\cC}\longrightarrow \Dual{\cB}
    \longrightarrow \Dual{\cA}\longrightarrow 0
  \]
  is exact.
  \item Applied to \eqref{eq:canonical-sequence}, the dual sequence is exact
  if and only if the natural obstruction homomorphism
  \[
    \CartierDual{\HHm(\cG)}\longrightarrow
    \ExtSheaf{2}{\HHz(\cG)}
  \]
  is zero.
\end{enumerate}
\end{lemma}

\begin{proof}
The first assertion is \cite[Proposition~3.18(b)]{BrochardDuality}.  The
second is the specialization of the same result to
\eqref{eq:canonical-sequence}, as described in
\cite[Remark~3.19]{BrochardDuality}.
\end{proof}

\subsection{Vanishing statements}

We use the notation
$\ExtSheaf{i}{H}$ for the $\fppf$ sheaf associated with
$T\mapsto\operatorname{Ext}^{i}_{T,\fppf}(H_T,\Gm)$.

\begin{proposition}\label{prop:ext-vanishing}
Let $S$ be a scheme.
\begin{enumerate}
  \item If $F$ is a finite flat commutative $S$-group scheme, then
  $\ExtSheaf{1}{F}=0$.  If $2$ is invertible on $S$, then also
  \[
    \ExtSheaf{2}{F}=\ExtSheaf{3}{F}=0.
  \]
  \item If $A$ is an abelian scheme, then
  $\CartierDual A=0$ and
  \[
    \ExtSheaf{1}{A}\simeq\AbDual A,
    \qquad
    \ExtSheaf{2}{A}=0.
  \]
\end{enumerate}
\end{proposition}

\begin{proof}
The first assertion is due to Breen
\cite{BreenExtensions,BreenVanishing}; see also
\cite[Theorems~11.1 and~11.2]{BrochardDuality}.  The Cartier-dual assertion
for an abelian scheme is \cite[Theorem~11.4]{BrochardDuality}.  The
Barsotti--Weil identification and the second-Ext vanishing are
\cite[Corollary~3.6 and Theorem~4.3]{RibeiroRosengarten}.
\end{proof}

More precisely, for every morphism $T\to S$,
Ribeiro and Rosengarten construct a functorial isomorphism
\begin{equation}\label{eq:global-ext2-abelian}
  \operatorname{Ext}^{2}_{T,\fppf}(A_T,\Gm)
  \simeq H^1_{\fppf}(T,A_T^\vee).
\end{equation}
These groups need not vanish; nevertheless, the fppf sheaf associated with
the left-hand presheaf is zero.

The classical duality theorem for abelian schemes gives
$\DoubleDual A\simeq A$; we use this identification without further comment.
When $S$ is locally noetherian, we refer to
\cite[Chapter~III, \S\S19--20]{Oort} for a proof.  
The arbitrary-base statement is  in
\cite[Expos\'{e}~VIII, \S3.2]{SGA7}.

\begin{lemma}\label{lem:brochard-quotient}
Let
\[
  0\longrightarrow F\longrightarrow G\longrightarrow B\longrightarrow0
\]
be an exact sequence of abelian $\fppf$ sheaves, where $F$ is a finite flat
group scheme and $B$ is an abelian scheme.  If
\[
  \delta\colon F^{\mathrm D}\longrightarrow B^\vee
\]
is the connecting homomorphism, then
\[
  \Dual G\simeq[B^\vee/F^{\mathrm D}],
\]
where $F^{\mathrm D}$ acts on $B^\vee$ by translation through $\delta$.
In particular, if $G$ is proper and flat, then $\Dual G$ is proper and flat
with finite diagonal.
\end{lemma}

\begin{proof}
The quotient presentation is \cite[Lemma~3.9]{BrochardDuality}; its only
vanishing hypothesis is $\ExtSheaf{1}{F}=0$.
Put $H=F^{\mathrm D}$, $X=B^\vee$, and $\cX=[X/H]$.  The action groupoid is
\[
  H\times_S X\rightrightarrows X.
\]
Its source and target maps are finite flat: the source is the projection,
and the target becomes the projection after applying the automorphism
$(h,x)\mapsto(h,h\cdot x)$.  Thus $\cX$ is algebraic and the quotient
morphism
\[
  p\colon X\longrightarrow\cX
\]
is representable, flat, locally of finite presentation, and surjective
\cite[Tags~06FH and~06FI]{Stacks}.  More precisely, $p$ is an $H$-torsor;
after any base change $T\to\cX$ it becomes an $H_T$-torsor over $T$, and is
therefore finite locally free.  Hence $p$ is a finite faithfully flat cover,
not necessarily a smooth atlas.

Since $X\to S$ is smooth, it is flat and locally of finite presentation.
Flatness and local finite presentation descend through $p$, so
$\cX\to S$ has both properties
\cite[Tags~06Q0 and~06Q9]{Stacks}.

We next compute the diagonal.  After the faithfully flat base change
$X\times_S X\to\cX\times_S\cX$, it is the action-groupoid morphism
\[
  j\colon H\times_S X\longrightarrow X\times_S X,
  \qquad (h,x)\longmapsto(x,h\cdot x).
\]
The first projection $X\times_S X\to X$ is separated, since $X\to S$ is
separated, while its composite with $j$ is the finite morphism
$H\times_S X\to X$.  It follows that $j$ is finite
\cite[Tag~081Z]{Stacks}.  The cartesian description of the diagonal of a
quotient stack identifies $j$ with the displayed base change of
$\Delta_{\cX/S}$ \cite[Tag~0DTY]{Stacks}.  Since finiteness is
$\fppf$-local on the target, the diagonal is finite
\cite[Tags~04XD and~0426]{Stacks}.  In particular, $\cX\to S$ is separated.

It remains to prove properness.  Since $p$ is surjective and $X\to S$ is
quasi-compact, $\cX\to S$ is quasi-compact \cite[Tag~050X]{Stacks}.
Together with local finite presentation, this shows that $\cX\to S$ is of
finite type.  Finally,
$X\to S$ is proper and hence universally closed.  Surjectivity of $p$
therefore makes $\cX\to S$ universally closed; combined with separatedness
and finite type, this proves properness by \cite[Tag~0CQK]{Stacks}.
\end{proof}

\subsection{Brochard's representability results}

We will make use of the following result from Brochard's work.

\begin{theorem}\label{thm:brochard-known}
Let $\cG$ be an algebraic commutative group stack that is proper, flat, and
finitely presented over $S$.
\begin{enumerate}
  \item The dual $\Dual{\cG}$ is algebraic and finitely presented over $S$,
  with affine diagonal and proper geometric fibres.
  \item If $\HHm(\cG)$ is flat, then
  $\HHm(\Dual{\cG})=\CartierDual{\HHz(\cG)}$ is finite over $S$.
  \item If $\cG$ is an algebraic space, then $\Dual{\cG}$ is flat over $S$.
\end{enumerate}
Formation of the dual commutes with arbitrary base change.
\end{theorem}

\begin{proof}
These are \cite[Remark~3.3 and Theorems~3.14--3.15]{BrochardDuality}.
\end{proof}

\section{Quasiabelian models over a discrete valuation ring}
\label{sec:quasiabelian}

\begin{definition}
Let $S$ be an integral scheme.  A \emph{quasiabelian $S$-group scheme} is a
proper flat finitely presented $S$-group scheme whose restriction to a dense
open subscheme of $S$ is an abelian scheme.  Over the spectrum of a discrete
valuation ring, this means that the generic fibre is an abelian variety.
\end{definition}

For an integral scheme $S$, we write $\generic$ for its generic point.  When
$S=\Spec R$ for an integral domain $R$, we identify
$\generic=\Spec\Frac(R)$.

This terminology agrees with \cite[Definition~13]{FakhruddinSrinivas}.  It
should not be confused with Brochard's term \emph{duabelian}: a duabelian
group scheme occurs in an exact sequence
\[
  0\longrightarrow A\longrightarrow G\longrightarrow F\longrightarrow0
\]
with $A$ abelian and $F$ finite flat
\cite[Definition~2.17]{BrochardDuality}.  The order of the two pieces is
essential over a discrete valuation ring.

We begin with two standard geometric facts.

\begin{lemma}\label{lem:projective-cm}
Let $R$ be a discrete valuation ring and let $S=\Spec R$.
\begin{enumerate}
  \item A separated group algebraic space locally of finite presentation over
  $S$ is a scheme.  If it is proper and flat over $S$, then it is projective
  over $S$.
  \item A flat finitely presented $S$-group scheme is Cohen--Macaulay.
\end{enumerate}
\end{lemma}

\begin{proof}
The representability assertion over a one-dimensional normal base is due to
Raynaud \cite[Theorem~3.3.1]{RaynaudSpecialisation}; see also
\cite[Theorem~4.B]{Anantharaman}.  A proper flat group scheme over a discrete
valuation ring is projective \cite[Remark~4.3.6]{RomagnyEffective}.

For the second assertion, every group scheme locally of finite type over a
field is Cohen--Macaulay
\cite[Expos\'{e}~VII$_B$, Corollaire~5.5.1]{SGA3I}.  Thus the flat morphism
to $S$ has Cohen--Macaulay fibres, and its total space is Cohen--Macaulay
because $S$ is regular \cite[Tags~045S and~045J]{Stacks}.
\end{proof}

\begin{lemma}\label{lem:auto-commutative}
Let $R$ be a discrete valuation ring and let $N$ be a proper flat finitely
presented group algebraic space over $R$ whose generic fibre is an abelian
variety.  Then $N$ is a commutative projective group scheme.
\end{lemma}

\begin{proof}
It is a projective scheme by \cref{lem:projective-cm}.  Its two multiplication
maps $m$ and $m\circ\tau$ agree on the abelian generic fibre.  The generic
fibre of the flat scheme $N\times_R N$ is schematically dense, and the target
is separated.  Hence the two maps agree everywhere.
\end{proof}

\subsection{The normalization theorem}

\begin{proposition}
\label{prop:normalization}
Let $R$ be a discrete valuation ring, and
let $N$ be a quasiabelian $R$-group scheme.  Let $A$ be the \Neron{} model of
the abelian variety $N_\generic$.  Then:
\begin{enumerate}
  \item $N_\generic$ has good reduction, so $A$ is an abelian scheme;
  \item the identity of the generic fibre extends uniquely to a homomorphism
  $u\colon A\To N$;
  \item $u$ is finite and is the normalization morphism of $N$.
\end{enumerate}
\end{proposition}

\begin{proof}
Let $p$ be the residue characteristic, with the convention $p=0$ in residue
characteristic zero, and choose a prime $\ell\neq p$.  Throughout, we work over
the original discrete valuation ring $R$.

For every $n\geq 1$, multiplication by $\ell^n$ on a flat group scheme over
$R$ is \'{e}tale; see
\cite[\S4, item~1]{FakhruddinSrinivas} and the SGA~3 result cited there.
Its kernel is quasi-finite, and since $N$ is proper, $N[\ell^n]$ is therefore
finite \'{e}tale over $R$.
Put $L=\Frac(R)$.  Fix a separable closure $L^{\mathrm{sep}}$, and let
$R^{\mathrm{sh}}$ be a strict henselization of $R$ inside
$L^{\mathrm{sep}}$, with fraction field $L^{\mathrm{sh}}$.  The associated
inertia subgroup is
$I_L=\operatorname{Gal}(L^{\mathrm{sep}}/L^{\mathrm{sh}})$.  Since
$N[\ell^n]$ is finite \'{e}tale over $R$, its base change to
$R^{\mathrm{sh}}$ is constant.  Thus $I_L$ acts trivially on
$N_\generic[\ell^n](L^{\mathrm{sep}})$ for every $n$, and therefore on
$T_{\ell}(N_\generic)$.  The \Neron--Ogg--Shafarevich good-reduction
criterion, in a form valid for arbitrary residue fields,
\cite[Theorem~3.2]{SilverbergZarhin} shows that $N_\generic$ has good
reduction.  Its \Neron{} model $A$ is therefore an abelian scheme.

The group-smoothening theorem
\cite[\S7.1, Theorem~5]{BLR} supplies a smooth separated finite-type
$R$-group scheme
$N^{\mathrm{sm}}$ and a homomorphism $N^{\mathrm{sm}}\To N$ inducing the
identity on the generic fibre and a bijection
\[
  N^{\mathrm{sm}}(R^{\mathrm{sh}})
  \xrightarrow{\ \sim\ }N(R^{\mathrm{sh}}).
\]
By the valuative criterion, properness and separatedness make the map
\[
  N(R^{\mathrm{sh}})\xrightarrow{\ \sim\ }N_\generic(L^{\mathrm{sh}})
\]
bijective.  The \Neron{} mapping criterion
\cite[\S7.1, Theorem~1]{BLR} therefore identifies $N^{\mathrm{sm}}$
with the \Neron{} model of $N_\generic$.
Identifying it with $A$ gives the required
homomorphism $u\colon A\To N$ directly over $R$.  This extension is unique:
$A$ is flat over $R$, so $A_\generic$ is schematically dense in $A$, while
$N$ is separated over $R$.

We prove that $u$ is finite.  Its restriction
\[
  A[\ell^n]\longrightarrow N[\ell^n]
\]
is a morphism of finite \'{e}tale $R$-group schemes and is the identity on the
generic fibre.  Restriction from finite \'{e}tale schemes over a normal
integral scheme to its generic point is fully faithful: a finite \'{e}tale
scheme is the normalization of the base in its generic algebra.  Thus the
displayed map is an isomorphism.  

Let $H=\Ker(u)$.  If a geometric fibre of $H$ had positive dimension, the
reduced neutral component of that fibre would be a positive-dimensional
abelian subvariety of the corresponding fibre of $A$.  It would contain
nontrivial $\ell$-torsion, contradicting the fact that
$A[\ell]\To N[\ell]$ is an isomorphism.  Thus $H$ is quasi-finite.  Every
nonempty geometric fibre of $u$ is a translate of a geometric fibre of $H$.
Hence $u$ is quasi-finite.  It is also proper, because $A$ is proper over $R$
and $N$ is separated over $R$.  Therefore $u$ is finite.

The schematic image of $u$ contains the generic fibre of $N$.  Flatness of
$N$ makes that fibre schematically dense, so $u$ is surjective.  Finally, $u$
is finite and birational and $A$ is normal.

We spell out the normalization argument.  The scheme $N$ is integral.  Indeed,
if $U=\Spec B$ is a nonempty affine open of $N$ and $\pi$ is a uniformizer,
flatness over $R$ makes $B\To B[1/\pi]$ injective.  The generic fibre is dense,
so $U_\generic$ is nonempty; as it is an open subscheme of the integral scheme
$N_\generic$,
$B[1/\pi]$ is a domain, and hence so is $B$.  Moreover, the dense generic
fibre is irreducible, so $N$ is irreducible.

It remains only to identify the already finite map $u$ with the
normalization; no finiteness theorem for normalization is needed.  Let
$U=\Spec C$ be an affine open of $N$ and write
$u^{-1}(U)=\Spec D$.  Then $D$ is finite over $C$, is normal, and has the
same fraction field.  Conversely, every element of that fraction field which
is integral over $C$ is integral over $D$ and therefore belongs to $D$.
Thus $D$ is the integral closure of $C$, and these affine identifications show
that $u$ is the normalization morphism.
\end{proof}

\begin{corollary}\label{cor:prime-to-p-torsion}
For every integer $n$ invertible in $R$, restriction of the normalization map
induces an isomorphism
\[
  A[n]\xrightarrow{\ \sim\ }N[n].
\]
\end{corollary}

\begin{proof}
Both sides are finite \'{e}tale over $R$, and the map is an isomorphism on the
generic fibre.  Apply the full-faithfulness argument used in the proof of
\cref{prop:normalization}.
\end{proof}

\begin{remark}\label{rem:normal-not-iso}
The normalization morphism in \cref{prop:normalization} need not be an
isomorphism.  In the nonsmooth examples of
\cite[Examples~4.11 and~4.12]{AchterCasalainaMartinWise}, its kernel cannot be
flat: a flat kernel that is trivial on the generic fibre would be trivial
everywhere.  Thus the finite flat kernel constructed below is not, in
general, the kernel of the normalization.
\end{remark}

\begin{remark}\label{rem:normalization-base-change}
The normalization is functorial and has a strong base change
property in this setting.  If $f:N\to N'$ is a homomorphism of quasiabelian
models, its generic fibre extends uniquely to a homomorphism of their abelian
normalizations
\[
  \widetilde f:A\longrightarrow A',\qquad
  u'\widetilde f=fu.
\]
If $R\to R'$ is a dominant local morphism of discrete valuation rings, then
$A_{R'}\to N_{R'}$ is finite and birational with normal source, and the same
affine argument as above identifies it with the normalization.  The
constructions below therefore commute with every such base change.
\end{remark}

\subsection{The polarization morphism}

Let $m_N$ be the group law of $N$, and let
$u\colon A\To N$ be the normalization morphism.  Given an invertible sheaf
$\cL$ on $N$, put $\cM=u^{*}\cL$ and consider on
$N\times_R A$ the invertible sheaf
\begin{equation}\label{eq:poincare-family}
  \cP_\cL
  :=(m_N\circ(\id_N\times u))^{*}\cL
     \otimes \operatorname{pr}_{A}^{*}\cM^{-1}.
\end{equation}
Viewed as a family of invertible sheaves on $A$ parametrized by $N$,
$\cP_\cL$ defines a morphism
\[
  q_\cL\colon N\longrightarrow \Pic_{A/R}.
\]
Its restriction to $\{0\}\times_R A$ is canonically trivial, so
$q_\cL$ sends the identity of $N$ to the identity of
$\Pic_{A/R}$.

Recall that $\Pic_{A/R}$ denotes the fppf relative Picard sheaf.  The
relative N\'{e}ron--Severi sheaf is the fppf quotient
\[
  \NS_{A/R}:=\Pic_{A/R}/\Pic^0_{A/R}.
\]
Put
\[
  H:=\underline{\operatorname{Hom}}_R(A,\AbDual A).
\]
If $\iota_A:A\To A^{\vee\vee}$ is the biduality isomorphism, duality
defines an involution $\dagger:H\To H$ by
$h^\dagger=h^\vee\circ\iota_A$.  Write
$H^{\mathrm{sd}}=\operatorname{Eq}(\id_H,\dagger)$.  The homomorphism
$[\cE]\mapsto\phi_\cE$ fits into an exact sequence of fppf
sheaves
\[
  0\longrightarrow \AbDual A\longrightarrow \Pic_{A/R}
  \longrightarrow H^{\mathrm{sd}}\longrightarrow 0;
\]
see \cite[\S3.1, equation~(13)]{PoonenRainsSelfCup}.  Hence
$\NS_{A/R}\simeq H^{\mathrm{sd}}$.

By \cite[Proposition~1.3.1 and its proof]{BertolinExtensions}, $H$ is
represented by a separated unramified $R$-group scheme locally of finite
presentation.  Its diagonal is closed by separatedness and open by
\cite[Tag~02GE]{Stacks}, so the
equalizer $H^{\mathrm{sd}}$ is an open-and-closed subgroup scheme of $H$.
Consequently, $\NS_{A/R}$ is represented by a separated unramified
$R$-group scheme locally of finite presentation.  This description commutes
with arbitrary base change.  Its fibre at a geometric point
$\bar s\To\Spec R$ is the discrete group scheme associated with the finitely
generated free group $\NS(A_{\bar s})$ of line-bundle classes modulo algebraic
equivalence.  If $A$ has positive relative dimension, these fibres are
infinite, so ``locally of finite presentation'' cannot in general be replaced
by ``of finite presentation''.

\begin{lemma}\label{lem:q-properties}
The morphism $q_\cL$ factors through
$\AbDual A=\Pic^0_{A/R}$ and is a homomorphism.  Moreover,
\begin{equation}\label{eq:q-polarization}
  q_\cL\circ u=\phi_{\cM},
  \qquad \cM=u^{*}\cL,
\end{equation}
where
$\phi_{\cM}(a)=t_a^{*}\cM\otimes \cM^{-1}$.
\end{lemma}

\begin{proof}
On the generic fibre, $u$ is the identity and
$q_{\cL,\generic}$ sends $a$ to
$t_a^{*}\cL_\generic\otimes\cL_\generic^{-1}$.  By the theorem of the square,
translation does not change the N\'{e}ron--Severi class of a line bundle on an
abelian variety; equivalently,
$t_a^{*}\cL_\generic\otimes\cL_\generic^{-1}$ is algebraically trivial.  Thus
$q_{\cL,\generic}$ is the usual homomorphism
$\phi_{\cL_\generic}\colon N_\generic\To\Pic^0_{N_\generic/\generic}$, and its
composite with
$\Pic_{A_\generic/\generic}\To\NS_{A_\generic/\generic}$ is zero.  Since $N$
is flat, its generic fibre is schematically dense, while $\NS_{A/R}$ is
separated.  The composite
$N\To\Pic_{A/R}\To\NS_{A/R}$ is therefore zero, proving the asserted
factorization.

The homomorphism property can likewise be checked on the schematically dense
generic fibre of $N\times_R N$, because $\AbDual A$ is separated.  Finally,
pulling \eqref{eq:poincare-family} back along $u\times\id_A$ gives the usual
translation-difference family for $\cM$.  This is precisely
\eqref{eq:q-polarization}.
\end{proof}

\begin{theorem}\label{thm:polarization-quotient}
In the setting of \cref{prop:normalization}, suppose that $\cL$ is
$R$-ample.  Then
\[
  q_\cL\colon N\longrightarrow \AbDual A
\]
is finite faithfully flat.  In particular,
\[
  0\longrightarrow F_\cL\longrightarrow N
  \overset{q_\cL}{\longrightarrow}\AbDual A
  \longrightarrow 0,
  \qquad F_\cL:=\Ker(q_\cL),
\]
is exact and $F_\cL$ is finite flat over $R$.
\end{theorem}

\begin{proof}
The pullback $\cM=u^{*}\cL$ is ample on $A$ as $u$ is affine.  Therefore
$\phi_{\cM}\colon A\To\AbDual A$ is a finite faithfully flat
polarization.  By \eqref{eq:q-polarization}, it factors as
\[
  A\overset{u}{\longrightarrow}N
   \overset{q_\cL}{\longrightarrow}\AbDual A.
\]
The morphism $q_\cL$ is proper as its source is proper over $R$ and
its target is separated over $R$.  For a geometric point $y$ of $\AbDual A$,
surjectivity of $u$ gives
\[
  u^{-1}\bigl(q_\cL^{-1}(y)\bigr)
  =\phi_{\cM}^{-1}(y).
\]
The right-hand side is finite, so $q_\cL^{-1}(y)$ is finite.  Thus
$q_\cL$ is proper and quasi-finite, hence finite.  It is surjective
because its composite with $u$ is surjective.

It remains to prove flatness.  By \cref{lem:projective-cm}, $N$ is
Cohen--Macaulay.  The scheme $\AbDual A$ is regular, since it is smooth over
the regular scheme $\Spec R$.  The integrality argument in the proof of
\cref{prop:normalization} shows that $N$ is integral.  Since
$q_\cL$ is surjective, it is dominant.  Thus, for $x\in N$ mapping
to $y\in\AbDual A$, the finite local homomorphism
\[
  C:=\cO_{\AbDual A,y}\longrightarrow
  B:=\cO_{N,x}
\]
is injective and integral.  The regular local ring $C$ is normal, so going
down, together with incomparability for integral extensions, gives
$\dim B=\dim C$.  Moreover, $\mathfrak m_C B$ is
$\mathfrak m_B$-primary.  A regular system of parameters of $C$ therefore
maps to a system of parameters of the Cohen--Macaulay local ring $B$, and
hence to a regular sequence.  The local flatness criterion
\cite[Tag~00R4]{Stacks} now implies that $C\To B$ is flat.  Hence $q_\cL$ is finite flat, and
surjectivity makes it faithfully flat.  Its kernel is the base change along
the identity section of $\AbDual A$, so it is finite flat over $R$.
\end{proof}

More generally we have:

\begin{proposition}\label{prop:conditional-polarization}
Let $S$ be an integral scheme, let $A$ be an abelian $S$-scheme, and let $N$
be a proper flat finitely presented $S$-group scheme.  Suppose that
$u:A\to N$ is a finite surjective homomorphism which is an isomorphism over a
dense open subscheme of $S$.  Every $S$-ample invertible sheaf $\cL$
on $N$ defines a finite faithfully flat homomorphism
\[
  q_\cL:N\longrightarrow A^\vee,
  \qquad q_\cL u=\phi_{u^*\cL}.
\]
\end{proposition}

\begin{proof}
The construction in \eqref{eq:poincare-family} is unchanged.  Schematic
density of the locus on which $u$ is an isomorphism proves factorization
through $\Pic^0_{A/S}$, the homomorphism property, and the displayed formula.
The formula and the polarization $\phi_{u^*\cL}$ show, as in the proof
of \cref{thm:polarization-quotient}, that $q_\cL$ is finite and
surjective.

It remains to check flatness.  For a geometric point $\bar s\to S$, the
finite image of the irreducible abelian variety $A_{\bar s}$ is all of
$N_{\bar s}$, so the latter is topologically irreducible and has dimension
$\dim A_{\bar s}$.  It is Cohen--Macaulay by
\cite[Expos\'{e}~VII$_B$, Corollaire~5.5.1]{SGA3I}.  Thus
$q_{\cL,\bar s}$ is a finite map from an equidimensional
Cohen--Macaulay scheme to the regular variety $A^\vee_{\bar s}$, with equal
local dimensions.  Miracle flatness makes it flat.  Since both source and
target are flat over $S$, the fibrewise flatness criterion
\cite[Tag~039A]{Stacks} makes $q_\cL$ flat.
\end{proof}

\subsection{The pushout description}

Retain the notation
\[
  \cM=u^*\cL,\qquad K=\Ker(\phi_{\cM}),\qquad
  F=\Ker(q_\cL).
\]
The identity $q_\cL u=\phi_{\cM}$ gives a homomorphism
\[
  h:=u|_K\colon K\longrightarrow F.
\]

Recall that a commutative square is \emph{bicartesian} if it is both
cartesian and cocartesian, that is, if it is simultaneously a pullback and a
pushout square.  These notions are taken in the category of abelian sheaves on
the big $\fppf$ site of the base (here $S=\Spec R$).

\begin{theorem}\label{thm:pushout}
There is a commutative diagram of abelian $\fppf$ sheaves with exact rows
whose left-hand square is bicartesian:
\begin{equation}\label{eq:pushout-diagram}
\begin{array}{ccccccccc}
0&\longrightarrow&K&\longrightarrow&A&\xrightarrow{\phi_{\cM}}&A^\vee
  &\longrightarrow&0\\
 &&\downarrow h&&\downarrow u&&\Vert\\
0&\longrightarrow&F&\longrightarrow&N&\xrightarrow{q_\cL}&A^\vee
  &\longrightarrow&0.
\end{array}
\end{equation}
Equivalently,
\[
  N\simeq A\amalg^K F\simeq(A\times_R F)/K,
\]
where $K$ is embedded by $k\mapsto(k,-h(k))$.  The group schemes $K$ and
$F$ have the same rank, and $h$ is an isomorphism on the generic fibre.
\end{theorem}

\begin{proof}
Since $F=\Ker(q_\cL)$ and $q_\cL u=\phi_{\cM}$, there is a
canonical identification
\[
  A\times_N F\simeq\Ker(\phi_{\cM})=K,
\]
so the left-hand square is cartesian.  To prove that it is cocartesian,
consider
\[
  \alpha:A\times_R F\longrightarrow N,
  \qquad (a,f)\longmapsto u(a)+f.
\]
We first observe that it is an $\fppf$ epimorphism.  Indeed, for a section $n$ of $N$, lift
$q_\cL(n)$ locally along the faithfully flat morphism
$\phi_{\cM}$.  If $a$ is a lift, then $n-u(a)$ belongs to $F$.  The kernel of
$\alpha$ consists precisely of the pairs $(k,-h(k))$ with $k\in K$.
Consequently, $\alpha$ is a $K$-torsor and represents the indicated quotient
and pushout.  On the generic fibre, $u$ is the identity and
$q_\cL$ is $\phi_{\cM}$; hence $h_\generic$ is an isomorphism.  Since $K$ and
$F$ are finite flat, their ranks are their common generic rank.
\end{proof}

Conversely, we have:

\begin{proposition}\label{prop:pushout-converse}
Let $A$ be an abelian scheme over a discrete valuation ring $R$, let
$\phi:A\to A^\vee$ be a polarization with kernel $K$, let $F$ be a finite
flat commutative $R$-group scheme, and let $h:K\to F$ be an isomorphism on
the generic fibre.  Then
\[
  (A\times_R F)/K,\qquad k\longmapsto(k,-h(k)),
\]
is represented by a proper flat quasiabelian $R$-group scheme with generic
fibre $A_\generic$.
\end{proposition}

\begin{proof}
The action is free, so the quotient is a scheme and
$A\times_R F$ is a finite faithfully flat $K$-torsor over it
\cite[Tag~07S7]{Stacks}.  Flatness and finite presentation descend along this
cover.  Since $A\times_R F$ is finite over the projective scheme $A$, it is
projective.  The diagonal argument gives separatedness of the quotient
\cite[Tag~09MQ]{Stacks}, while universal closedness descends through the
surjective cover \cite[Tag~03GN]{Stacks}; hence the quotient is proper.  It is
projective by \cref{lem:projective-cm}.  The group law descends because the
embedded copy of $K$ is a subgroup.  On the generic fibre, the map
\[
  A_\generic\times_\generic F_\generic\longrightarrow A_\generic,
  \qquad(a,f)\longmapsto a+h_\generic^{-1}(f),
\]
identifies the quotient with $A_\generic$.
\end{proof}

\begin{remark}\label{rem:pushout-classification}
  The
pair $(F,h)$ depends on the chosen ample line bundle, and arbitrary pushout
data do not include a prescribed line bundle descending from $A$.
Nevertheless, in the derived category of abelian $\fppf$ sheaves we have 
a quasi-isomorphism
\[
  \Delta_N:=\operatorname{Cone}(A\xrightarrow{u}N)
  \simeq\operatorname{Cone}(K\xrightarrow{h}F).
\]
\end{remark}

\begin{remark}\label{rem:q-functorial}
The construction is functorial.  With the notation of
\cref{rem:normalization-base-change}, an invertible sheaf $\cL'$ on
$N'$ satisfies
\[
  q_{f^*\cL'}=\widetilde f^{\vee}\circ q_{\cL'}\circ f.
\]
Moreover, when $N'=N$, one has $q_{\cL\otimes\cL'}=
q_\cL+q_{\cL'}$.  Both assertions follow on the generic
fibre from the corresponding facts for polarizations and extend by schematic
density.  The quotient, its kernel, and the pushout datum commute with every
dominant local extension of discrete valuation rings.
\end{remark}

For a morphism $f:X\to Y$ with locally Noetherian fibres, recall that $f$ is
\emph{Gorenstein} if it is flat and every fibre of $f$ is a Gorenstein scheme
\cite[Tags~0C04 and~0C05]{Stacks}.

\begin{corollary}\label{cor:gorenstein}
The morphism $q_\cL:N\to A^\vee$ is finite, flat, and Gorenstein.
Consequently, the structure morphism $f:N\to\Spec R$ is Gorenstein.  The
relative dualizing sheaf $\omega_{N/R}$ of $f$ is translation invariant and
is noncanonically trivial.  If $g=\dim N_\generic$, a choice of
trivialization gives a derived self-duality
\[
  \RHom_R\bigl(\mathbf R\Gamma(N,\cO_N),R\bigr)
  \simeq \mathbf R\Gamma(N,\cO_N)[g].
\]
\end{corollary}

\begin{proof}
By \cref{prop:conditional-polarization}, $q_\cL$ is an $\fppf$
morphism.  After base change along $q_\cL:N\to A^\vee$, it becomes
the projection
\[
  \operatorname{pr}_2:F\times_R N\longrightarrow N.
\]
Indeed, $(z,n)\mapsto(n+z,n)$ identifies $F\times_R N$ with
$N\times_{A^\vee}N$.

Let $H=\Gamma(F,\cO_F)$.  Since $R$ is local, $\Pic(R)=0$, and hence
the finite locally free Hopf $R$-algebra $H$ is Frobenius
\cite[Corollary~1]{Pareigis}; thus
\[
  \operatorname{Hom}_R(H,R)\simeq H
\]
as $H$-modules.  The left-hand side is the relative dualizing module of
$F/R$, so $F\to\Spec R$ is Gorenstein
\cite[Tags~0BUK and~0C16]{Stacks}.  Hence its base change
$\operatorname{pr}_2$ is Gorenstein. Being Gorenstein is $\fppf$ local on the target, \cite[Tag~0C0A]{Stacks}.
So $q_\cL$ is Gorenstein.

Since $A^\vee\to\Spec R$ is smooth, it is Gorenstein, and therefore the
composite $f:N\xrightarrow{q_\cL}A^\vee\to\Spec R$ is Gorenstein
\cite[Tags~0C15 and~0C11]{Stacks}.

  Put $S=\Spec R$, let
$p_i:N\times_S N\to N$ be the two projections, let $m=m_N$ be the group law,
and let $e:S\to N$ be the identity section.  Since $q_\cL$ is finite
flat and $A^\vee/S$ is smooth of relative dimension $g$, the Gorenstein
morphism $f:N\to S$ has pure relative dimension $g$.  Thus its relative
dualizing complex has the form
\[
  \omega_{N/S}^{\bullet}:=f^!\cO_S\simeq\omega_{N/R}[g],
\]
where $\omega_{N/R}$ is invertible
\cite[Tags~0BV8 and~0C08]{Stacks}.

Formation of the relative dualizing complex for a proper flat finitely
presented morphism commutes with arbitrary base change
\cite[Tag~0B6S]{Stacks}.  The morphism $p_2:N\times_S N\to N$ is the base
change of $f$ along $f$, with the map to the original copy of $N$ given by
$p_1$.  Hence
\[
  \omega_{(N\times_S N)/N}^{\bullet}
  \simeq \mathbf Lp_1^*\omega_{N/S}^{\bullet}
  \simeq p_1^*\omega_{N/R}[g].
\]
Here $p_1$ is flat, being a base change of $f$, so its derived pullback is the
ordinary pullback.

Now consider the universal translation automorphism over the second factor
\[
  T=(m,p_2):N\times_S N\longrightarrow N\times_S N,
  \qquad (x,a)\longmapsto(x+a,a).
\]
Its inverse is $(x,a)\mapsto(x-a,a)$, and $p_2\circ T=p_2$.  Thus $T$ is an
automorphism over $N$ for the structure morphism $p_2$.  Applying the same
base-change statement to the cartesian square defined by $T$ over
$\id_N$ gives
\[
  T^*\omega_{(N\times_S N)/N}^{\bullet}
  \simeq\omega_{(N\times_S N)/N}^{\bullet}.
\]
Since $p_1\circ T=m$, taking cohomology in degree $-g$ gives
\[
  m^*\omega_{N/R}\simeq p_1^*\omega_{N/R}.
\]
This is universal translation invariance: after any base change $U\to S$,
pullback along $x\mapsto(x,a)$ for $a\in N(U)$ gives
$t_a^*\omega_{N_U/U}\simeq\omega_{N_U/U}$.

Finally, let
\[
  i=(e\circ f,\id_N):N\longrightarrow N\times_S N,
\]
whose image is $\{0\}\times_S N$.  Since $m\circ i=\id_N$ and
$p_1\circ i=e\circ f$, pulling the preceding isomorphism back by $i$ gives
\[
  \omega_{N/R}\simeq f^*e^*\omega_{N/R}.
\]
The $R$-module $e^*\omega_{N/R}$ is invertible and hence free of rank one
because $R$ is local.  A choice of generator gives a noncanonical
trivialization $\omega_{N/R}\simeq\cO_N$.  With this choice,
Grothendieck duality gives
\[
  \RHom_R\bigl(\mathbf R\Gamma(N,\cO_N),R\bigr)
  \simeq \mathbf R\Gamma(N,f^!\cO_S)
  \simeq \mathbf R\Gamma(N,\cO_N)[g]
\]
\cite[Corollary~(4.2.2)]{LipmanDuality}.
\end{proof}

\subsection{The special fibre}

Let $k$ be the residue field, let $\bar k$ be an algebraic closure of
characteristic $p>0$, and put
\[
  G=N_{\bar s}.
\]

\begin{theorem}\label{thm:special-defect}
The group scheme $G$ is connected and irreducible, and $G_{\mathrm{red}}$ is
an abelian variety.  The specialisation of the normalization map factors as
an isogeny
\[
  A_{\bar s}\xrightarrow{v}G_{\mathrm{red}}
  \lhook\joinrel\longrightarrow G.
\]
The quotient $C:=G/G_{\mathrm{red}}$ is a finite connected group scheme, and
\[
  0\longrightarrow G_{\mathrm{red}}\longrightarrow G
  \longrightarrow C\longrightarrow0
\]
is exact.  If
\[
  e(N):=\deg(v),
\]
then
\begin{equation}\label{eq:defect-rank}
  e(N)=\operatorname{rk}(C).
\end{equation}
In particular, $e(N)$ is a power of $p$, and the fundamental cycle of $G$ is
\begin{equation}\label{eq:defect-cycle}
  [G]=e(N)[G_{\mathrm{red}}].
\end{equation}
\end{theorem}

\begin{proof}
The map $A_{\bar s}\to G$ is finite and surjective.  Thus
$G$ is connected and irreducible.  Since $A_{\bar s}$ is
reduced, its map to $G$ factors through $G_{\mathrm{red}}$.  Over the perfect
field $\bar k$, the reduction of a group scheme is a closed smooth subgroup by
Cartier's theorem \cite[Expos\'{e}~VI$_A$, \S0]{SGA3I}.  It is proper,
connected, and topologically the image of $A_{\bar s}$, so it is an abelian
variety.  The induced map $v:A_{\bar s}\to G_{\mathrm{red}}$ is the base
change of the finite map $A_{\bar s}\to G$ along
$G_{\mathrm{red}}\to G$; it is finite and surjective, hence an isogeny.

The quotient by the closed smooth subgroup $G_{\mathrm{red}}$ is
representable, and $G\to C$ is a $G_{\mathrm{red}}$-torsor
\cite[Expos\'{e}~V, Corollaire~10.1.3]{SGA3I}.  It is proper and
zero-dimensional, hence finite.  Its underlying topological space is one
point, so it is connected.

Base change on abelian $\fppf$ sheaves is exact, see \cite[03I4, 04JC]{Stacks}.  By the characterization of
cartesian and cocartesian squares in an abelian category
\cite[Tag~08N2]{Stacks}, the left-hand square in
\eqref{eq:pushout-diagram} remains bicartesian after base change to $\bar s$.
Thus we obtain a diagram with exact rows
\[
\begin{array}{ccccccccc}
0&\longrightarrow&K_{\bar s}&\longrightarrow&A_{\bar s}
  &\xrightarrow{\phi_{\cM_{\bar s}}}&A^\vee_{\bar s}
  &\longrightarrow&0\\
&&\downarrow h_{\bar s}&&\downarrow u_{\bar s}&&\Vert\\
0&\longrightarrow&F_{\bar s}&\longrightarrow&G
  &\xrightarrow{q_{\cL,\bar s}}&A^\vee_{\bar s}
  &\longrightarrow&0.
\end{array}
\]
Apply the snake lemma to this diagram.  Since the right-hand vertical map is
the identity, it gives canonical isomorphisms
\[
  \Ker(h_{\bar s})\simeq\Ker(u_{\bar s}),\qquad
  \Coker(h_{\bar s})\simeq\Coker(u_{\bar s}).
\]
The map $u_{\bar s}$ factors as
\[
  A_{\bar s}\xrightarrow{v}G_{\mathrm{red}}
  \lhook\joinrel\longrightarrow G.
\]
Since the second map is a monomorphism, $\Ker(u_{\bar s})=\Ker(v)$.
Notice also that $\Ker(v)\subseteq K_{\bar s}$: if $v(a)=0$, then
$\phi_{\cM_{\bar s}}(a)=q_{\cL,\bar s}(u_{\bar s}(a))=0$.
Moreover, $v$ is an isogeny and hence an $\fppf$ epimorphism.  Therefore the
image of $u_{\bar s}$ as an $\fppf$ sheaf is $G_{\mathrm{red}}$, and
\[
  \Coker(u_{\bar s})=G/G_{\mathrm{red}}=C.
\]
Substituting these identifications into the canonical kernel--cokernel
sequence for $h_{\bar s}$ gives an exact sequence of finite group schemes
\begin{equation}\label{eq:defect-exact}
  0\longrightarrow\Ker(v)\longrightarrow K_{\bar s}
  \xrightarrow{h_{\bar s}}F_{\bar s}\longrightarrow C
  \longrightarrow0.
\end{equation}
Since $K$ and $F$ have equal rank, \eqref{eq:defect-exact} gives
\[
  \operatorname{rk}(C)=\operatorname{rk}\Ker(v)=\deg(v).
\]
A finite connected group scheme over a perfect field of characteristic $p$
has $p$-power rank
\cite[\S14.4, first corollary, p.~112]{Waterhouse}.
Finally, flat pullback of cycles along the $G_{\mathrm{red}}$-torsor
$G\to C$ gives \eqref{eq:defect-cycle}.
\end{proof}

Thus the degeneration has no additional components, toric part, or
positive-dimensional affine part: it is a purely nilpotent, $p$-primary
thickening of an abelian variety.  In residue characteristic zero, Cartier's
theorem makes the finite connected group $C$ \'{e}tale, hence trivial; the
same argument shows directly that every quasiabelian model is an abelian
scheme.

For a proper flat finitely presented morphism $f:X\to S$, recall that $f$ is
\emph{cohomologically flat in degree zero} if formation of
$f_*\cO_X$ commutes with arbitrary base change: for every morphism
$g:S'\to S$, writing $f':X\times_S S'\to S'$ for the projection, the
canonical map
\[
  g^*f_*\cO_X\longrightarrow
  f'_*\cO_{X\times_S S'}
\]
is an isomorphism \cite[\S~1.4]{RaynaudSpecialisation}.  Thus, for $N/R$, this
says equivalently that for every $R$-algebra $R'$ the canonical map
\[
  \Gamma(N,\cO_N)\otimes_R R'
  \longrightarrow \Gamma(N_{R'},\cO_{N_{R'}})
\]
is an isomorphism.

\begin{corollary}\label{cor:defect-cohomology}
Let $g=\dim G_{\mathrm{red}}$.  For every $i\geq0$,
\begin{equation}\label{eq:defect-cohomology}
  \dim_{\bar k}H^i(G,\cO_G)
  =e(N)\binom gi.
\end{equation}
Moreover,
\[
  C\simeq\Spec\Gamma(G,\cO_G),\qquad
  \Gamma(N,\cO_N)=R.
\]
Consequently, $N/R$ is cohomologically flat in degree zero if and only if
$e(N)=1$.
\end{corollary}

\begin{proof}
Let $\pi:G\to C$ be the $G_{\mathrm{red}}$-torsor.  There is an $\fppf$
cover $U\to C$ such that $G\times_CU\to U$ identifies with the projection
$G_{\mathrm{red}}\times_{\bar k}U\to U$.  By
\cite[Lemma~80.11.7]{Stacks}, the map $\pi$ is itself an $\fppf$ cover, so
one may take $U=G$.  Flat base change and descent show that
$R^i\pi_*\cO_G$ is locally free of rank
$\binom gi$; here we use the standard description of the coherent cohomology
algebra of an abelian variety \cite[\S13]{Mumford}.  Since $C$ is finite local
of length $e(N)$, taking global sections gives
\eqref{eq:defect-cohomology}.  The case $i=0$ also gives
$\pi_*\cO_G=\cO_C$ and the displayed description of $C$.

Properness makes $\Gamma(N,\cO_N)$ finite over $R$.  Integrality of
$N$ and density of the generic fibre embed this ring into
$\Gamma(N_\generic,\cO_{N_\generic})=\Frac(R)$.  It therefore lies in the
integral closure of $R$ in $\Frac(R)$, which is $R$, and equality follows.
The degree-zero base-change map on the special fibre is consequently an
isomorphism exactly when
$h^0(G,\cO_G)=e(N)$ is one.  The cohomology and base change criterion
gives the last assertion, \cite[Tags~0E62 and~02KH]{Stacks}.
\end{proof}

\begin{corollary}\label{cor:defect-rigidity}
The following conditions are equivalent:
\begin{enumerate}
  \item $N$ is an abelian scheme;
  \item $N$ is smooth;
  \item $N$ is normal;
  \item $u:A\to N$ is an isomorphism;
  \item $G$ is reduced;
  \item $N$ is cohomologically flat in degree zero;
  \item $e(N)=1$;
  \item $C$ is the trivial group scheme;
  \item $h:K\to F$ is an isomorphism for one, equivalently every, ample
  $\cL$;
  \item $N^{\mathrm D}$ is trivial;
  \item $\Dual N$ is an algebraic space.
\end{enumerate}
\end{corollary}

\begin{proof}
The first four conditions are equivalent because $u$ is the normalization,
an abelian scheme is smooth and normal, and a smooth model is normal.  The
next four equivalences follow from
\cref{thm:special-defect,cor:defect-cohomology}; the trivial group scheme has
rank one.  The pushout of \cref{thm:pushout} is $A$ precisely when $h$ is an
isomorphism.  Every character of $G$ kills the abelian variety
$G_{\mathrm{red}}$ and hence factors uniquely through $C$, so
\[
  G^{\mathrm D}\simeq C^{\mathrm D}.
\]
Cartier duality is faithful on finite group schemes and commutes with base
change.  Moreover, $N^{\mathrm D}$ is finite by
\cref{thm:brochard-known}; its generic fibre is trivial, and when its special
fibre is trivial its augmentation ideal vanishes by Nakayama's lemma.  This
proves the tenth equivalence.  Finally, the inertia of $\Dual N$ is
$N^{\mathrm D}$ by \eqref{eq:hminus-dual}; it is trivial exactly when the dual
is an algebraic space.
\end{proof}

\subsection{Polarizations and \texorpdfstring{$\ell$}{ell}-adic cohomology}

Let $\cL_{\mathrm{red}}$ be the restriction of $\cL_{\bar s}$
to $G_{\mathrm{red}}$.

\begin{proposition}\label{prop:polarization-defect}
With $\cM=u^*\cL$, one has
\[
  \cM_{\bar s}=v^*\cL_{\mathrm{red}},
  \qquad
  \phi_{\cM_{\bar s}}=v^\vee\circ\phi_{\cL_{\mathrm{red}}}\circ v,
\]
and therefore
\begin{equation}\label{eq:polarization-defect}
  \deg(\phi_{\cM})=e(N)^2\deg(\phi_{\cL_{\mathrm{red}}}),
  \qquad
  \chi(\cM)=e(N)\chi(\cL_{\mathrm{red}}).
\end{equation}
In particular, $e(N)^2$ divides the degree of every polarization represented
by an ample line bundle extending to $N$.  Writing $(d_1,\ldots,d_g)$ for a
polarization type, with $\chi(\cM)=d_1\cdots d_g$, one has
\[
  e(N)\mid d_1\cdots d_g,
  \qquad
  v_p(e(N))\leq v_p(d_1\cdots d_g).
\]
An extending polarization of degree prime to $p$, and in particular an
extending principal polarization, forces $N=A$.
\end{proposition}

\begin{proof}
The first identities follow by restricting $u$ and the translation/difference
construction to the reduced special fibre.  Degrees of polarizations and
Euler characteristics are constant in the abelian scheme $A/R$.  For the
isogeny $v$, the standard degree and Euler characteristic formulas give
\eqref{eq:polarization-defect}.  Since
$\chi(\cM)=d_1\cdots d_g$ and
$\deg(\phi_{\cM})=(d_1\cdots d_g)^2$, the divisibility statements follow.
If the degree is prime to $p$, then the $p$-power $e(N)$ is one, and
\cref{cor:defect-rigidity} applies.
\end{proof}

We have the following analogue of the main result of \cite{FakhruddinSrinivas}. In that 
paper the focus was on quasi-reductive group schemes.

\begin{proposition}\label{prop:prime-to-p-cohomology}
For every prime $\ell\neq p$, there are canonical isomorphisms
\[
  H^i_{\etale}(G,\mathbf Z_\ell)
  \simeq H^i_{\etale}(G_{\mathrm{red}},\mathbf Z_\ell)
  \xrightarrow[\sim]{v^*}H^i_{\etale}(A_{\bar s},\mathbf Z_\ell).
\]
Together with smooth proper base change for $A$, these form a canonical
zigzag to the generic-fibre cohomology.  If the original residue field $k$ is
finite, the $k$-schemes $N_s$, $(N_s)_{\mathrm{red}}$, and $A_s$ have the same
zeta function.
\end{proposition}

\begin{proof}
The nilpotent immersion $G_{\mathrm{red}}\hookrightarrow G$ induces an
equivalence of \'{e}tale topoi.  The isogeny $v$ has $p$-power degree by
\cref{thm:special-defect}, so it induces an isomorphism on integral
$\ell$-adic cohomology for $\ell\neq p$.  Over a finite field, nilpotent
thickenings have the same points over every finite extension, and isogenous
abelian varieties have the same zeta function.
\end{proof}

\begin{remark}\label{rem:conductor}
The coherent sheaf
\[
  \cQ_N:=u_*\cO_A/\cO_N
\]
is supported on the special fibre and vanishes exactly when $N$ is abelian.
Its annihilator
\[
  \mathfrak c_N:=\operatorname{Ann}_{\cO_N}(\cQ_N)
  =\{f\in\cO_N\mid
      f\,u_*\cO_A\subseteq\cO_N\}
\]
is the conductor ideal; equivalently, it is the largest ideal of
$\cO_N$ that is also an ideal of $u_*\cO_A$.

The group scheme $A$ acts on $N$ by $(a,n)\mapsto u(a)+n$ and on itself by
translation.  The morphism $u$ is equivariant for these actions: for every
$R$-scheme $T$ and every $a\in A(T)$,
\[
  u_T\circ t_a=t_{u_T(a)}\circ u_T.
\]
Consequently, the inclusion $\cO_N\hookrightarrow u_*\cO_A$ is
$A$-equivariant, $\cQ_N$ inherits an $A$-linearization, and both
$\mathfrak c_N$ and $\operatorname{Supp}(\cQ_N)$ are invariant under
translation by $A$.  After base change to $\bar s$, the isogeny
$v:A_{\bar s}\to G_{\mathrm{red}}$ is surjective on $\bar k$-points, while
$|G_{\mathrm{red}}|=|G|$.  Hence $A_{\bar s}(\bar k)$ acts transitively on the
closed points of $G$.  It follows that the support of
$\cQ_N|_G$, when nonempty, contains every closed point and is therefore
all of $|G|$.

\end{remark}

\begin{proof}[Proof of \cref{thm:intro-structure}]
Automatic commutativity is \cref{lem:auto-commutative}.  Combine
\cref{prop:normalization,lem:q-properties,thm:polarization-quotient,thm:pushout}.
\end{proof}

\section{The dual of a quasiabelian model}\label{sec:local-duality}

Throughout this section, $R$ is an arbitrary discrete valuation ring and $N$
is quasiabelian.  Choose an ample invertible sheaf $\cL$ on $N$ and
retain the notation $\cM$, $K$, $F$, and $h$ of \cref{thm:pushout}.

\begin{theorem}\label{thm:dual-quasiabelian}
For the chosen polarization quotient there is a natural presentation
\begin{equation}\label{eq:dual-quasiabelian-quotient}
  \Dual N\simeq[A/F^{\mathrm D}].
\end{equation}
The action is by translation through
\begin{equation}\label{eq:dual-action}
  \delta:F^{\mathrm D}\xrightarrow{h^{\mathrm D}}K^{\mathrm D}
  \xrightarrow{\iota_{\cM}}A,
\end{equation}
where $\iota_{\cM}$ is the connecting map obtained by dualizing the polarization
sequence.  In particular,
\[
  \HHm(\Dual N)=\Ker(h^{\mathrm D}),\qquad
  \HHz(\Dual N)=\Coker(\delta),
\]
and there is an exact sequence
\begin{equation}\label{eq:ext1-quasiabelian}
  0\longrightarrow\Coker(h^{\mathrm D})
  \longrightarrow\ExtSheaf{1}{N}
  \longrightarrow A^\vee\longrightarrow0.
\end{equation}
The stack $\Dual N$ is algebraic, proper, flat, and finitely presented over
$R$, with finite diagonal.  These assertions require no assumption on $2$.
\end{theorem}

\begin{proof}
Apply \cref{lem:brochard-quotient} to
\[
  0\longrightarrow F\longrightarrow N\longrightarrow A^\vee
  \longrightarrow0.
\]
The boundary map in that lemma is $\delta$.  Indeed, the pushout diagram
identifies it with $h^{\mathrm D}$ followed by the boundary map for
$0\to K\to A\to A^\vee\to0$.  Dualizing the latter sequence gives
\[
  0\longrightarrow K^{\mathrm D}\xrightarrow{\iota_{\cM}}A
  \longrightarrow A^\vee\longrightarrow0.
\]
Thus $\iota_{\cM}$ is a monomorphism, which gives the formulas for $\HHm$ and
$\HHz$ and the exact sequence \eqref{eq:ext1-quasiabelian}.  The geometric
properties follow from \cref{lem:brochard-quotient}.
\end{proof}

The quotient is represented by the two-term complex
\begin{equation}\label{eq:dual-quasiabelian}
  [F^{\mathrm D}\xrightarrow{\delta}A].
\end{equation}
The degreewise short exact sequence of complexes gives a distinguished
triangle
\[
  A\longrightarrow\Dual N\longrightarrow\BG{F^{\mathrm D}}
  \longrightarrow A[1].
\]
It is not in general a short exact sequence of group stacks; its cohomology
sequence is the one displayed in \cref{thm:dual-quasiabelian}.
For genuinely stacky inputs, the remaining higher-Ext issue comes from the
finite flat part.

\begin{proposition}\label{prop:ext2-quasiabelian}
There is a natural monomorphism
\[
  \ExtSheaf{2}{N}\longrightarrow\ExtSheaf{2}{F}.
\]
In particular, if $2\in R^\times$, then
\begin{equation}\label{eq:ext2-quasiabelian}
  \ExtSheaf{2}{N}=0.
\end{equation}
\end{proposition}

\begin{proof}
The long Ext sequence of $0\to F\to N\to A^\vee\to0$ contains
\[
  \ExtSheaf{2}{A^\vee}\longrightarrow
  \ExtSheaf{2}{N}\longrightarrow
  \ExtSheaf{2}{F}.
\]
The left-hand term vanishes by \cref{prop:ext-vanishing}, which gives the
monomorphism.  If $2\in R^\times$, the right-hand term vanishes by the same
proposition.
\end{proof}

\begin{corollary}\label{cor:thickness-stackiness}
On the geometric special fibre,
\[
  \HHm(\Dual G)=G^{\mathrm D}\simeq C^{\mathrm D}.
\]
Its rank is $e(N)$, whereas the generic stabilizer of $\Dual N$ is trivial.
If $N$ is singular, the inertia of $\Dual N$ is finite but not flat over
$R$.  If $2\in R^\times$, there is also an exact sequence
\[
  0\longrightarrow\BG{C^{\mathrm D}}\longrightarrow\Dual G
  \longrightarrow (G_{\mathrm{red}})^\vee\longrightarrow0.
\]
\end{corollary}

\begin{proof}
Every character of $G$ factors uniquely through $C$, as in the proof of
\cref{cor:defect-rigidity}.  This proves the first assertion in every
characteristic.  The last sequence follows by dualizing
$0\to G_{\mathrm{red}}\to G\to C\to0$ and using the finite-group higher-Ext
vanishing from \cref{prop:ext-vanishing}.
\end{proof}

Thus duality exchanges vertical nilpotent thickness with vertical stackiness.
In particular, it preserves properness and flatness but not flatness of
inertia.

The quotient presentation also gives the required separation statement.

\begin{corollary}\label{cor:separation}
The structural morphism $\Dual N\To\Spec R$ and the diagonal of $\Dual N$ are
proper.  In particular, if $V$ is a discrete valuation ring over $R$ and a
multiplicative invertible sheaf on $N_V$ is generically trivial, then its
generic multiplicative trivialization extends uniquely over $V$.
\end{corollary}

\begin{proof}
Properness is part of \cref{thm:dual-quasiabelian}, and the diagonal is in fact
finite.  The sheaf of multiplicative isomorphisms from the given object of
$\Dual N(V)$ to the neutral object is the pullback of that finite diagonal.
It is finite over $V$.  A section over $\Frac(V)$ extends uniquely over the
integrally closed valuation ring $V$ by the valuative criterion for finite
morphisms.
\end{proof}

\section{Finite extensions of abelian schemes}\label{sec:devissage}

In this section we prove the main structure theorems for 
flat, proper, finitely presented group algebraic spaces.

\begin{proposition}\label{prop:generic-devissage}
Let $R$ be a discrete valuation ring and let $P$ be a proper flat finitely
presented commutative group algebraic space over $R$.  There is an exact
sequence
\begin{equation}\label{eq:generic-devissage}
  0\longrightarrow N\longrightarrow P\longrightarrow Q\longrightarrow0,
\end{equation}
where $N$ is quasiabelian and $Q$ is finite flat.
\end{proposition}

\begin{proof}
By \cref{lem:projective-cm}, $P$ is a projective group scheme.  Over
$\generic$, the structure theorem for proper group schemes gives
\[
  0\longrightarrow B_\generic\longrightarrow P_\generic
  \longrightarrow Q_\generic\longrightarrow0,
\]
with $B_\generic$ an abelian variety and $Q_\generic$ finite
\cite[Expos\'{e}~VI$_B$, \S12]{SGA3I}.  Let $N$ be the schematic closure of
$B_\generic$ in $P$.  It is a flat closed subgroup, proper and finitely
presented over $R$, and its generic fibre is $B_\generic$.

The action groupoid $N\times_RP\rightrightarrows P$ satisfies the hypotheses
of \cite[Expos\'{e}~V, Th\'{e}or\`eme~7.1]{SGA3I}: its projections are proper
and flat, $P$ is projective over $R$, and its relation morphism
$N\times_RP\to P\times_RP$ is a closed immersion.  Hence the $\fppf$
quotient $Q=P/N$ is represented by a finitely presented $R$-group scheme,
and $P\to Q$ is proper and faithfully flat.  The relation morphism is the
pullback of the diagonal of $Q$, so faithfully flat descent of closed
immersions shows that $Q$ is separated.  Flatness of $Q$ over $R$ descends
from $P$.  Since $P\to Q$ is surjective and $P$ is proper over $R$, the
morphism $Q\to\Spec R$ is universally closed; being separated and finitely
presented, it is proper.  The flatness of $P$ and $N$ makes their fibre
dimensions constant.  Since their generic dimensions are equal, the
dimension formula for the $N_{\bar s}$-torsor
$P_{\bar s}\to Q_{\bar s}$ gives $\dim Q_{\bar s}=0$.  Thus $Q$ is
quasi-finite and proper, hence finite; its flatness makes it finite flat.
\end{proof}

\begin{theorem}\label{thm:group-devissage}
Let $R$ and $P$ be as in \cref{prop:generic-devissage}.  There is an exact
sequence
\begin{equation}\label{eq:direct-extension}
  0\longrightarrow E\longrightarrow P\xrightarrow{f}B
  \longrightarrow0,
\end{equation}
where $E$ is finite flat and $B$ is an abelian scheme.  Moreover,
\begin{equation}\label{eq:dual-general-group}
  \Dual P\simeq[B^\vee/E^{\mathrm D}].
\end{equation}
Consequently, $\Dual P$ is algebraic, proper, flat, and finitely presented
over $R$, with finite diagonal, in every characteristic.
\end{theorem}

\begin{proof}
Take \eqref{eq:generic-devissage}, and let $u:A\to N$ be the normalization
from \cref{prop:normalization}.  Since $N$ is projective, choose an $R$-ample
invertible sheaf $\cL$ on $N$, put $B=A^\vee$, and let
$q=q_\cL:N\to B$ be the polarization quotient from
\cref{thm:polarization-quotient}; thus
$q\circ u=\phi_{u^*\cL}$.  Put
$m=\operatorname{rk}(Q)$.  Deligne's rank-annihilation theorem for finite
locally free commutative group schemes gives $[m]_Q=0$, see
\cite[\S1, page 4]{TateOort}.  Hence $[m]_P$ factors uniquely through $N$:
\[
  [m]_P:P\xrightarrow{r}N\lhook\joinrel\longrightarrow P.
\]
Set $f=q\circ r$.  On $N$ one has
\[
  g:=f|_N=q\circ[m]_N=[m]_B\circ q,
\]
so $g$ is finite faithfully flat.

Let $E=\Ker(f)$ and $K_0=\Ker(g)$.  There is an exact sequence
\begin{equation}\label{eq:kernel-collapse}
  0\longrightarrow K_0\longrightarrow E\longrightarrow Q
  \longrightarrow0.
\end{equation}
To see surjectivity on the right, lift a section of $Q$ $\fppf$-locally to
$p\in P$, then choose $n\in N$ with $g(n)=-f(p)$.  The section $p+n$ lies in
$E$ and maps to the prescribed section of $Q$; the kernel is  $K_0$.
Thus $E\to Q$ is a $K_0$-torsor and $E$ is finite flat.  Since the restriction
$g$ is an $\fppf$ epimorphism, so is $f$, and $P\to B$ is the $E$-torsor
associated with its kernel.  This proves \eqref{eq:direct-extension}.

The quotient presentation \eqref{eq:dual-general-group} is
\cref{lem:brochard-quotient}.  Its finite faithfully flat cover is $B^\vee$,
which gives properness and flatness.  Algebraicity and finite presentation
follow from \cref{thm:brochard-known}, and the quotient presentation has finite
diagonal.
\end{proof}

The construction of \eqref{eq:direct-extension} uses neither an Ext vanishing
nor an assumption on $2$.  It is the discrete valuation analogue of Raynaud's
\cite[Proposition~3.11]{BrochardDuality}.

\begin{corollary}\label{cor:ext2-general-group}
There is a natural monomorphism
\[
  \ExtSheaf{2}{P}\longrightarrow\ExtSheaf{2}{E}.
\]
In particular, if $2\in R^\times$, then $\ExtSheaf{2}{P}=0$.
\end{corollary}

\begin{proof}
The long Ext sequence of \eqref{eq:direct-extension} contains
\[
  \ExtSheaf{2}{B}\longrightarrow
  \ExtSheaf{2}{P}\longrightarrow
  \ExtSheaf{2}{E}.
\]
The left-hand term vanishes by \cref{prop:ext-vanishing}, which gives the
monomorphism.  If $2\in R^\times$, the right-hand term vanishes by the same
proposition.
\end{proof}

\begin{proof}[Proof of \cref{thm:intro-group-structure}]
This is \cref{thm:group-devissage}.
\end{proof}

We now pass to group stacks.  For a proper flat commutative group stack
$\cG$ with finite flat inertia, the group sheaf of isomorphism classes
$P=\HHz(\cG)$ is a proper flat finitely presented group algebraic space by
the proof of \cite[Proposition~2.14]{BrochardDuality}; over a discrete
valuation ring it is a scheme by
\cref{lem:projective-cm}.

\begin{proposition}\label{prop:local-stack-structure}
Let $R$ be a discrete valuation ring and let $\cG$ be a proper flat finitely
presented algebraic commutative group stack over $R$ with finite flat inertia.
There is an exact sequence
\[
  0\longrightarrow\cK\longrightarrow\cG\longrightarrow B
  \longrightarrow0,
\]
where $B$ is an abelian scheme and both
\[
  \HHm(\cK)=\HHm(\cG),\qquad \HHz(\cK)=E
\]
are finite flat group schemes.
\end{proposition}

\begin{proof}
Put $P=\HHz(\cG)$ and take the quotient $P\to B$ of
\cref{thm:group-devissage}.  Define $\cK$ as the kernel of the composite
$\cG\to P\to B$.  The cohomology sequence gives the displayed
identifications.
\end{proof}

\begin{theorem}\label{thm:local-stack}
Let $R$ be a discrete valuation ring with $2\in R^{\times}$.  Let $\cG$ be a
proper flat finitely presented algebraic commutative group stack over $R$, and
assume that $H:=\HHm(\cG)$ is finite flat.  Then $\Dual{\cG}$ is algebraic,
proper, flat, and finitely presented over $R$, with finite diagonal.  More
precisely, the canonical sequence dualizes to an exact sequence
\begin{equation}\label{eq:local-stack-dual}
  0\longrightarrow \Dual P\longrightarrow \Dual{\cG}
  \longrightarrow \CartierDual H\longrightarrow0,
  \qquad P:=\HHz(\cG).
\end{equation}
\end{theorem}

\begin{proof}
By \cref{thm:group-devissage,cor:ext2-general-group},
$\ExtSheaf{2}{P}=0$ and $\Dual P$ is proper and flat.  The exactness of
\eqref{eq:local-stack-dual} follows from part~(ii) of
\cref{lem:dual-exactness} applied to
\eqref{eq:canonical-sequence}; here
$\Dual{\BG H}=\CartierDual H$.  The morphism
$\Dual{\cG}\to\CartierDual H$ is a $\Dual P$-torsor, hence $\fppf$-locally a
projection.  The right-hand term is finite flat.  It follows
from \eqref{eq:local-stack-dual} that $\Dual{\cG}$ is proper and flat.
Algebraicity and finite presentation are supplied by
\cref{thm:brochard-known}.  Finally,
$\HHm(\Dual{\cG})=\CartierDual P$ is finite by the same theorem.  Since
$\Dual{\cG}$ is proper, its diagonal is proper.  Its geometric fibres are
empty or torsors under the finite inertia group, so the diagonal is
quasi-finite and therefore finite.
\end{proof}

\section{The global duality theorem}\label{sec:global}

We first globalize the characteristic-free result for group algebraic spaces.

\begin{theorem}\label{thm:global-group-duality}
Let $S$ be a scheme and let $P$ be a proper flat finitely presented
commutative group algebraic space over $S$.  Then $\Dual P$ is algebraic,
proper, flat, and finitely presented over $S$, with finite diagonal.  The
evaluation morphism
\[
  e_P:P\longrightarrow\DoubleDual P
\]
is an isomorphism.
\end{theorem}

\begin{proof}
Suppose first that $S$ is locally noetherian.  By
\cref{thm:brochard-known}, the dual is algebraic, flat, and finitely presented,
with affine diagonal and proper geometric fibres.  Its inertia is finite by
\cite[Theorem~3.15]{BrochardDuality}.

Let $V$ be a discrete valuation ring with a morphism $\Spec V\to S$.  The
base change $P_V$ is a proper flat group algebraic space and hence a scheme by
\cref{lem:projective-cm}.  The direct structure theorem and quotient formula
of \cref{thm:group-devissage} show that $\Dual{P_V}$ is proper over $V$.
Duality commutes with base change, so every discrete-valuation-ring test for
$\Dual P\to S$ satisfies existence and uniqueness.  The  valuative
criteria \cite[Tags~0E80 and~0H2C]{Stacks} make $\Dual P$ separated and
universally closed, hence proper.  Its diagonal is proper and has fibres
which are empty or torsors under finite inertia; it is therefore quasi-finite
and thus finite.

For arbitrary $S$, the assertion is local on the base.  Write $S=\Spec B$
and express $B$ as a filtered colimit of finitely generated
$\mathbf Z$-subalgebras.  Noetherian approximation for algebraic spaces and
their morphisms descends $P$, its group structure, and, at a finite stage,
properness and flatness
\cite[Appendix~B, Propositions~B.2--B.3]{RydhApproximation}.  Apply the
locally noetherian case and pull back, using compatibility of the dual with
base change.

It remains to prove biduality.  The representability theorem applied again
makes $\DoubleDual P$ algebraic and finitely presented.  For every geometric
point $\bar s\to S$, the proper group scheme $P_{\bar s}$ is duabelian in
Brochard's terminology and is dualizable without any hypothesis on $2$ by
\cite[Theorem~4.13(i)]{BrochardDuality}.  Thus $e_P$ is an isomorphism on
every geometric fibre.

Put $K=\HHm(\DoubleDual P)$.  Since $\DoubleDual P$ is algebraic, $K$ is a
commutative group algebraic space locally of finite type over $S$
\cite[Tag~050Q]{Stacks}.  Formation of inertia commutes with base change
\cite[Tag~06PQ]{Stacks}.  Hence, for every geometric point $\bar s\to S$,
compatibility of duality with base change and the preceding fibrewise
biduality give
\[
  K_{\bar s}\simeq \HHm(\DoubleDual{P_{\bar s}})=0.
\]
The unramified locus commutes with arbitrary base change
\cite[Tag~05W2]{Stacks}, so $K\to S$ is unramified.  Its identity section
$S\to K$, being a base change of the diagonal, is therefore an open immersion
by \cite[Tag~05W1]{Stacks}.  Since $K_{\bar s}=0$, this section is surjective
on points over every algebraically closed field, and hence is surjective by
\cite[Tag~0487]{Stacks}.  An open subspace whose underlying set is all of
$|K|$ is $K$ itself
\cite[Tag~03BZ]{Stacks}.  Thus $S\to K$ is an isomorphism and $K=0$.

Translation in the group stack $\DoubleDual P$ now identifies every
stabilizer with $K$.  Thus $\DoubleDual P$ has trivial inertia and is an
algebraic space by \cite[Tag~04SZ]{Stacks}.  Finally,
\cite[Lemma~3.20]{BrochardDuality}, applied directly to
\[
  e_P:P\longrightarrow\DoubleDual P,
\]
shows that $e_P$ is an isomorphism: its source is flat and finitely
presented, its target is locally of finite type, and it is an isomorphism on
every geometric fibre.
\end{proof}

\begin{proof}[Proof of \cref{thm:intro-group-duality}]
This is \cref{thm:global-group-duality}.
\end{proof}

We now return to the case of a stack.  We first prove flatness of the dual;
\cref{prop:global-properness} then proves properness.

\begin{proposition}\label{prop:global-flatness}
Let $S$ be a locally noetherian scheme with
$2\in\Gamma(S,\cO_S)^{\times}$.
Let $\cG$ be as in \cref{thm:intro-brochard}.  Then $\Dual{\cG}$ is flat over
$S$.
\end{proposition}

\begin{proof}
By \cref{thm:brochard-known}, the dual is algebraic and finitely presented.
The infinitesimal flatness criterion
\cite[Expos\'{e}~IV, Th\'{e}or\`eme~5.6]{SGA1}, in the form used in the proof of
\cite[Theorem~3.14(3)]{BrochardDuality}, reduces the question to base changes
$T=\Spec R\To S$, where $R$ is an Artinian local ring with residue field
$k$.  Formation of the
dual commutes with base change, and flatness is fpqc local on the base.  We
may therefore make the faithfully flat local extension constructed in the
proof of \cite[Corollary~3.12]{BrochardDuality} and assume that the residue
field of $R$ is algebraically closed.
Write
\[
  H=\HHm(\cG_T),\qquad P=\HHz(\cG_T).
\]
The group algebraic space $P$ is a scheme.  Indeed, its special fibre is a
separated group algebraic space locally of finite type over the residue
field, hence is a scheme \cite[Tag~0B8F]{Stacks}.  Since the maximal ideal of
$R$ is nilpotent, $P_k\To P$ is a thickening, and an
algebraic-space thickening of a scheme is a scheme
\cite[Tag~05ZR]{Stacks}.

Raynaud's d\'{e}vissage over an Artinian local ring gives an exact sequence
\begin{equation}\label{eq:artin-devissage}
  0\longrightarrow F\longrightarrow P\longrightarrow A\longrightarrow0,
\end{equation}
where $F$ is finite flat and $A$ is an abelian scheme
\cite[Proposition~3.11]{BrochardDuality}.  Since $2$ is invertible on $R$,
\cref{prop:ext-vanishing} gives
\[
  \ExtSheaf{2}{F}=0,
  \qquad
  \ExtSheaf{2}{A}=0.
\]
The long Ext sequence therefore gives $\ExtSheaf{2}{P}=0$.  Hence the
obstruction homomorphism
\[
  \CartierDual H\longrightarrow\ExtSheaf{2}{P}
\]
vanishes.  By part~(ii) of
\cref{lem:dual-exactness}, there is an exact sequence
\[
  0\longrightarrow\Dual P\longrightarrow\Dual{\cG_T}
  \longrightarrow\CartierDual H\longrightarrow0.
\]
The morphism $\Dual{\cG_T}\to\CartierDual H$ is a $\Dual P$-torsor, hence
$\fppf$-locally a projection.
The group stack $\Dual P$ is flat by
\cite[Theorem~3.14(3)]{BrochardDuality}, and $\CartierDual H$ is finite flat.
Thus $\Dual{\cG_T}$ is flat over $T$.  The Artinian criterion now proves that
$\Dual{\cG}$ is flat over $S$.
\end{proof}

\begin{proposition}\label{prop:global-properness}
Under the hypotheses of \cref{prop:global-flatness}, the stack
$\Dual{\cG}$ has finite diagonal and is proper over $S$.
\end{proposition}

\begin{proof}
Put $P=\HHz(\cG)$.  By \eqref{eq:hminus-dual},
$\HHm(\Dual{\cG})=\CartierDual P$, and this inertia group is finite by
\cite[Theorem~3.15]{BrochardDuality}, since $\HHm(\cG)$ is flat.

Let $V$ be a discrete valuation ring equipped with a morphism
$\Spec V\To S$.  Duality commutes with base change, and
\cref{thm:local-stack} applies directly over $V$.  Thus
$\Dual{\cG}\times_S\Spec V$ is proper over $V$, so every
discrete-valuation-ring test satisfies existence and uniqueness up to
isomorphism.
The stack $\Dual{\cG}$ is finitely presented and has affine, hence
quasi-compact and separated, diagonal by \cref{thm:brochard-known}.
The Noetherian valuative criteria
\cite[Tags~0E80 and~0H2C]{Stacks} therefore show that it is separated and
universally closed, respectively.  Hence it is proper over $S$.

Its diagonal is now proper.  Its geometric fibres are empty or torsors under
the finite inertia group $\CartierDual P$, so the diagonal is quasi-finite.  A
proper quasi-finite morphism is finite.
\end{proof}

\begin{theorem}\label{thm:global-main}
Let $S$ be a scheme on which $2$ is invertible.  Let $\cG$ be a proper, flat,
finitely presented algebraic commutative group stack over $S$ with finite flat
inertia.  Then $\Dual{\cG}$ is algebraic, proper, flat, and finitely presented
over $S$, and its diagonal is finite.
\end{theorem}

\begin{proof}
When $S$ is locally noetherian, algebraicity and finite presentation follow
from \cref{thm:brochard-known}, flatness from
\cref{prop:global-flatness}, and properness and finiteness of the diagonal from
\cref{prop:global-properness}.

For an arbitrary base the assertion is Zariski local, so write
$S=\Spec B$.  Since $2\in B^{\times}$, express $B$ as a filtered colimit of
finitely generated $\mathbf Z[1/2]$-subalgebras $B_i$, and put
$S_i=\Spec B_i$.  Then $S=\varprojlim S_i$, with affine transition maps and
each $S_i$ noetherian.  By
\cite[Appendix~B, Proposition~B.2(ii)]{RydhApproximation}, after increasing
$i$ there is an algebraic stack $\cG_i$ of finite presentation over $S_i$
whose pullback is $\cG$.

The Hom-limit equivalence
\cite[Appendix~B, Proposition~B.2(i)]{RydhApproximation} lets us descend,
after a further increase of $i$, the multiplication, unit, inverse, and the
associativity, unit, inverse, and commutativity $2$-isomorphisms.  The same
equivalence makes the finitely many coherence diagrams commute at a finite
stage.  Thus $\cG_i$ may be taken to be a strictly commutative group stack.
Its inertia morphism pulls back to that of $\cG$.  By
\cite[Appendix~B, Proposition~B.3]{RydhApproximation}, after increasing $i$
again, $\cG_i\To S_i$ is proper and flat and its inertia is finite and flat.
Apply the locally noetherian case to $\cG_i$.  Formation of the dual commutes
with arbitrary base change by \cref{thm:brochard-known}; pulling back from
$S_i$ gives the result over $S$.
\end{proof}

\begin{corollary}[Brochard Duality]\label{cor:brochard}
Under the hypotheses of \cref{thm:global-main}, the evaluation morphism
\[
  e_{\cG}\colon\cG\longrightarrow\DoubleDual{\cG}
\]
is an isomorphism.  Thus Brochard's conjecture holds.
\end{corollary}

\begin{proof}
By \cref{thm:global-main}, part~(i) of
\cref{thm:intro-brochard} holds.  Brochard proves that this implies
dualizability \cite[Remark~4.14]{BrochardDuality}.  Briefly, the same
representability theorem applies to the bidual.  Over every algebraically
closed field the evaluation morphism is an isomorphism by
\cite[Theorem~4.11]{BrochardDuality}.  The fibrewise isomorphisms on
$\HHm$ and $\HHz$, together with flatness and finite presentation, globalize
by \cite[Lemma~3.20 and Proposition~4.5]{BrochardDuality}.  Hence
$e_{\cG}$ is an isomorphism.
\end{proof}

\begin{proof}[Proof of \cref{thm:intro-brochard}]
The first assertion is \cref{thm:global-main}, and the second is
\cref{cor:brochard}.
\end{proof}

\section{Further applications of the structure theorem}\label{sec:applications}

\subsection{Fourier--Mukai duality}

The starting point for this theory is \cite{MukaiDuality}.  Laumon
extended this result to generalized $1$-motives \cite{LaumonFourier}.
Donagi--Pantev and Arinkin treated gerbes and group stacks, including the
Fourier--Mukai property for ``good'' commutative group stacks
\cite[Section~4 and Appendix~A]{DonagiPantev}.  Laumon's generalized
$1$-motive theory is developed over a field of characteristic zero, and the
gerby Fourier--Mukai equivalences proved by Donagi--Pantev are over the
complex numbers.  Thus the equivalence theorems just cited from Laumon and
Donagi--Pantev are characteristic-zero results.  Polishchuk introduced kernel
algebras and proved unbounded quasi-coherent and bounded coherent equivalences
for finite-flat generalized $1$-motives
\cite[Theorem~4.4.2]{PolishchukKernel}.  Chen--Zhu proved the corresponding
bounded quasi-coherent equivalence for Beilinson $1$-motives over a
noetherian base \cite[Theorem~A.4.6]{ChenZhu}.  More recently,
\cite{DhillonNasserdenMukai} established stable $\infty$-categorical Mukai duality for
tame abelian stacks.

Chen--Zhu's theorem already has no characteristic restriction, but has strong
geometric restrictions.  By
\cite[Definition~A.4.1 and Lemma~A.4.4]{ChenZhu}, a Beilinson $1$-motive is
fpqc-locally of the form
\[
  A\times\BG{G}\times\Gamma,
\]
where $A$ is an abelian scheme, $G$ is of multiplicative type, and $\Gamma$
is a locally constant group with finitely generated fibres.  If such a
Picard stack is represented by a proper finitely presented group scheme,
then triviality of its inertia forces $G$ to be trivial and quasi-compactness
forces $\Gamma$ to be finite.  It is therefore fpqc-locally a finite disjoint
union of abelian schemes, and in particular is smooth.

The proper flat group schemes treated here need not be smooth or normal.
Theorem~\ref{thm:group-devissage} instead gives every proper flat finitely
presented commutative group algebraic space over a discrete valuation ring an
exact sequence
\[
  0\longrightarrow E\longrightarrow P\longrightarrow B\longrightarrow0
\]
with $E$ an arbitrary finite flat group scheme and $B$ an abelian scheme.
In positive residue characteristic this includes singular quasiabelian models
with nonreduced special fibre and finite flat group schemes with connected
special fibre.  These nonsmooth group schemes are not Beilinson $1$-motives.
Moreover,
\cref{lem:brochard-quotient} identifies the actual Picard dual as
$\Dual P\simeq[B^\vee/E^{\mathrm D}]$, a quotient which may have
non-linearly reductive inertia.  The new input relative to Chen--Zhu is thus
the structure theorem which brings these additional, potentially non-tame
objects within Polishchuk's finite-flat kernel formalism.  It also yields the
unbounded quasi-coherent and bounded coherent equivalences below, whereas
\cite[Theorem~A.4.6]{ChenZhu} states a bounded quasi-coherent equivalence.

For a possibly singular quasiabelian model,
\cref{thm:dual-quasiabelian} describes the dual as the quotient of the abelian
normalization by a finite-flat translation action.
For an abelian stack $\cA$, its dual $G=\Dual{\cA}$ is a duabelian group
scheme and biduality identifies $\cA$ with $\Dual G$.  Applying
\cref{thm:fourier-mukai} to $G$ therefore gives an analogue of
\cite{DhillonNasserdenMukai} without tameness for the derived categories of
the abelian categories of sheaves.

We retain the notation $\DQcoh{X}$ and $\DbCoh{X}$ introduced in the
introduction.  In particular, the first notation means the unbounded derived
category of the abelian category $\operatorname{Qcoh}(X)$.

\begin{theorem}\label{thm:fourier-mukai}
Let $R$ be a discrete valuation ring and let $P$ be a proper flat finitely
presented commutative $R$-group algebraic space.  There are exact equivalences
\begin{equation}\label{eq:fourier-mukai}
  \DQcoh{P}\simeq\DQcoh{\Dual P},
  \qquad
  \DbCoh{P}\simeq\DbCoh{\Dual P}.
\end{equation}
These are the generalized Fourier--Mukai transforms determined by the
Poincar\'{e} kernel data.

More precisely, choose the sequence from \cref{thm:group-devissage},
\[
  0\longrightarrow E\longrightarrow P\longrightarrow B\longrightarrow0,
\]
and let $\delta:E^{\mathrm D}\to B^\vee$ be its connecting homomorphism.
Under the identification
\[
  \Dual P\simeq[B^\vee/E^{\mathrm D}],
\]
the bounded coherent equivalence takes the form
\begin{equation}\label{eq:fourier-mukai-equivariant}
  \DbCoh{P}
  \simeq \DbCoh{[B^\vee/E^{\mathrm D}]}
  \simeq
  \mathbf D^{\mathrm b}
  \bigl(\operatorname{Coh}^{E^{\mathrm D}}(B^\vee)\bigr),
\end{equation}
where $E^{\mathrm D}$ acts by translation through $\delta$.  Note that no assumption on
the invertibility of
$2$ or on linear reductivity of $E$ or $E^{\mathrm D}$ is required.
\end{theorem}

Before giving the proof we will recall the pertinent parts of Polishchuk's theory, see \cite{PolishchukKernel}.

\medskip
\noindent\textit{Polishchuk's kernel-algebra formalism.}
  Let $S$
be noetherian and semi-separated (affine diagonal), and let $X,Y,Z$ be noetherian, flat,
semi-separated $S$-schemes.  For kernels $\cF$ on $X\times_S Y$ and $\cG$
on $Y\times_S Z$, their convolution is
\[
  \cF\circ_Y\cG
  :=\mathbf Rp_{13*}\bigl(
      \mathbf Lp_{12}^{*}\cF\otimes^{\mathbf L}
      \mathbf Lp_{23}^{*}\cG\bigr).
\]
A complex $\cK$ of quasi-coherent sheaves on $X\times_SY$ is
\emph{right convolution-flat over $Y$} if, for every
$\cH\in\operatorname{Qcoh}(Y)$, the complex
\[
  \cK\circ_Y\cH
  :=\mathbf Rp_{X*}\bigl(
      \cK\otimes^{\mathbf L}\mathbf Lp_Y^*\cH\bigr)
\]
is concentrated in degree zero, where $p_X$ and $p_Y$ are the projections
from $X\times_SY$.  It is \emph{left convolution-flat over $X$} if, for
every $\cH\in\operatorname{Qcoh}(X)$, the complex
\[
  \cH\circ_X\cK
  :=\mathbf Rp_{Y*}\bigl(
      \mathbf Lp_X^*\cH\otimes^{\mathbf L}\cK\bigr)
\]
is concentrated in degree zero.  We simply say \emph{convolution-flat over
the indicated factor} when the side is clear.

A kernel algebra on $X$ is an algebra object for this convolution: a kernel
$\cA$ on $X\times_SX$ with an associative multiplication
$\cA\circ_X\cA\to\cA$ and a two-sided unit
$\Delta_*\cO_X\to\cA$, where $\Delta_*\cO_X$ is the identity
kernel.  It is \emph{pure} if it is represented by a quasi-coherent sheaf and
is convolution-flat over both factors.  Equivalently, for every
$\cF\in\operatorname{Qcoh}(X)$, both $\cA\circ_X\cF$ and
$\cF\circ_X\cA$ are concentrated in degree zero.  A left module is a quasi-coherent
sheaf $\cF$ on $X$ with an associative unital action
$\cA\circ_X\cF\to\cF$.  A pure kernel algebra is \emph{finite} if it is
coherent and its support is proper over the first factor; its coherent
modules are then the modules whose underlying sheaf is coherent
\cite[Section~2.2]{PolishchukKernel}.

\noindent\textit{Definition.}
Let $\cA$ be a pure kernel algebra on $X$.  If $\cM$ is a left
$\cA$-module which is convolution-flat over $X$, and $\cN$ is a right
$\cA$-module, their balanced convolution is
\[
  \cN\circ_{\cA}\cM
  :=\operatorname{coeq}\bigl(
    \cN\circ_X\cA\circ_X\cM
    \rightrightarrows \cN\circ_X\cM\bigr),
\]
where the two arrows are induced by the two module actions.  A left
$\cA$-module $\cM$ is \emph{convolution-flat over $\cA$} if it is
convolution-flat over $X$ and the functor
$\cN\mapsto\cN\circ_{\cA}\cM$ on right $\cA$-modules is
exact.  If $Y$ is flat over $S$, a kernel $\cF$ on $X\times_SY$ with a left
$\cA$-action is \emph{convolution-flattening over $\cA$} if it is
convolution-flat over $Y$ and, for every flat quasi-coherent sheaf
$\cV$ on $Y$, the left $\cA$-module $\cF\circ_Y\cV$ is
convolution-flat over $\cA$.

A pure kernel algebra $\cA$ is \emph{left admissible} if there is a chain-map
quasi-isomorphism $\cA\to\cM^\bullet$ of complexes of $\cA$-bimodules
such that $\cM^\bullet$ is bounded and every $\cM^n$, after
forgetting its right $\cA$-action, is convolution-flattening over $\cA$ (with
$Y=X$).  If $\sigma:X\times_SX\to X\times_SX$ interchanges the factors, the
opposite kernel algebra is $\cA^{\mathrm{opp}}=\sigma_*\cA$, with multiplication
induced by reversing convolution.  The algebra $\cA$ is \emph{right
admissible} if $\cA^{\mathrm{opp}}$ is left admissible, and it is
\emph{admissible} if it is both left and right admissible
\cite[Section~2.4]{PolishchukKernel}.  Every pure kernel algebra of affine
type is admissible \cite[Lemma~2.4.7]{PolishchukKernel}; the finite-action
algebras below are pure and of affine type
\cite[Proposition~3.6.2]{PolishchukKernel}.

Here is the basic source of such algebras.  Let a finite flat group scheme
$G$ act on $X$, and write
\[
  \gamma:G\times_SX\longrightarrow X\times_SX,
  \qquad (g,x)\longmapsto(gx,x).
\]
If $p_G:G\times_SX\to G$ is the projection and $\omega_{G/S}$ is the
relative dualizing sheaf, the action kernel algebra is
\[
  \cA_X^G=\gamma_*p_G^*\omega_{G/S}.
\]
Its modules are precisely $G$-equivariant quasi-coherent sheaves on $X$
\cite[Theorem~3.3.1 and Corollary~3.3.2]{PolishchukKernel}.  More generally,
a line-valued $1$-cocycle is a line bundle $\cL$ on $G\times_SX$ with a
unit and associative isomorphisms
\[
  \cL_{g_1g_2,x}\simeq
  \cL_{g_1,g_2x}\otimes\cL_{g_2,x}.
\]
It defines the twisted action algebra
\[
  \cA_X^G(\cL)
  =\gamma_*\bigl(p_G^*\omega_{G/S}\otimes\cL\bigr).
\]
Modules over this algebra are $\cL$-twisted $G$-equivariant
quasi-coherent sheaves, and the analogous assertion holds for coherent
sheaves \cite[Corollary~3.4.5]{PolishchukKernel}.  In particular, for the
trivial cocycle,
\[
  \operatorname{Qcoh}([X/G])
    \simeq \cA_X^G\text{-}\operatorname{mod},
  \qquad
  \operatorname{Coh}([X/G])
    \simeq \cA_X^G\text{-}\operatorname{mod}_{\mathrm c}.
\]
No exactness of the invariants functor, and hence no linear reductivity of
$G$, enters this description.

In Polishchuk's terminology, a finite-flat generalized $1$-motive, or an
orbi-abelian scheme, has a presentation
\[
  M=[G_0\longrightarrow Q],
  \qquad
  0\longrightarrow L\longrightarrow Q\longrightarrow A\longrightarrow0,
\]
in degrees $-1,0$, where $G_0$ and $L$ are finite flat and $A$ is an
abelian scheme.  It represents the quotient stack $[Q/G_0]$, and
\[
  \operatorname{Qcoh}(M)=\operatorname{Qcoh}^{G_0}(Q),
  \qquad
  \operatorname{Coh}(M)=\operatorname{Coh}^{G_0}(Q),
\]
where $G_0$ acts through $G_0\to Q$ by translation.  These categories depend
only on $M$, not on its presentation.  For a fixed presentation, the
explicit two-term partner used in the kernel-algebra proof is
\[
  M^\dagger=[L^{\mathrm D}\longrightarrow Q^\dagger],
  \qquad
  0\longrightarrow G_0^{\mathrm D}\longrightarrow Q^\dagger
    \longrightarrow A^\vee\longrightarrow0.
\]
The extension $Q^\dagger$ is classified by the composite $G_0\to Q\to A$.
The arrow $L^{\mathrm D}\to Q^\dagger$ lies over the homomorphism
$L^{\mathrm D}\to A^\vee$ classifying $0\to L\to Q\to A\to0$; its lift is
supplied by the compatible Poincar\'{e} trivialization
\cite[Propositions~4.3.1 and~4.3.2(iii)]{PolishchukKernel}.
When the full derived dual is two-term, $M^\dagger$ represents Polishchuk's
$\RHom(M,\Gm)[1]$.  The proof of
\cite[Theorem~4.4.2]{PolishchukKernel} represents the categories attached to
$M$ and $M^\dagger$ by finite kernel algebras on $A$ and $A^\vee$ and proves
that those algebras are Fourier--Mukai dual.

We spell out the version of these results needed here.  For an extension
\[
  0\longrightarrow E\longrightarrow P\xrightarrow{\pi}B
  \longrightarrow0,
  \qquad H=E^{\mathrm D},
\]
\cite[Lemma~3.5.3]{PolishchukKernel} identifies its class, under finite
Cartier duality, with a homomorphism $\delta:H\to B^\vee$.  If $\sP_B$ is
the normalized Poincar\'{e} bundle on $B\times_RB^\vee$, then
\[
  \cL_\delta
    =(\operatorname{pr}_B,\delta\operatorname{pr}_H)^*\sP_B
      \quad\text{on }H\times_RB
\]
is a $1$-cocycle for the trivial $H$-action on $B$.  Push the identity
kernel on $P$ down along $\pi\times\pi$:
\[
  \cA(0\to P)
    :=(\pi\times\pi)_*\Delta_{P*}\cO_P.
\]
The affine push-forward equivalence of
\cite[Lemma~2.2.6]{PolishchukKernel} identifies its modules with
$\operatorname{Qcoh}(P)$, and likewise in the coherent case.  Polishchuk's
Cartier-duality comparison gives
\[
  \cA(0\to P)\simeq\cA_B^H(\cL_\delta)
\]
\cite[Corollary~3.5.5]{PolishchukKernel}.  His
\cite[Corollary~3.7.4]{PolishchukKernel} says that classical
Fourier--Mukai transform for $B$ and $B^\vee$ carries this twisted algebra to
the untwisted algebra $\cA_{B^\vee}^H$, where $H$ acts on $B^\vee$ by
translation through $\delta$.  The modules on the second side are therefore
\[
  \operatorname{Qcoh}^{H}(B^\vee)
    \simeq\operatorname{Qcoh}([B^\vee/H]),
\]
and similarly for coherent sheaves.

Finally, if mutually inverse adjoint Fourier--Mukai kernels are each a sheaf
up to a shift and carry one admissible kernel algebra to another, they induce
mutually inverse equivalences between the unbounded derived module categories
\cite[Theorem~2.5.1(i)]{PolishchukKernel}.
When the schemes are proper, the Fourier kernels have finite Tor-dimension,
and the algebras are finite, part~(iii) of that theorem restricts the
equivalence to the bounded derived categories of coherent modules.  In particular, we use the explicit
kernel-algebra comparison in the proof of Polishchuk's Theorem~4.4.2, not an
assertion that the full complex $\RHom(P,\Gm)[1]$ is two-term in
characteristic $2$.
\medskip

\begin{proof}[Proof of \cref{thm:fourier-mukai}]
By \cref{lem:projective-cm}, $P$ is a scheme.  Put $H=E^{\mathrm D}$ and
$M=[0\to P]$, with terms in degrees $-1$ and $0$.  The explicit finite
Fourier partner is
\[
  M^\dagger=[H\xrightarrow{\delta}B^\vee].
\]
Thus
\[
  \operatorname{Qcoh}(M)\simeq\operatorname{Qcoh}(P),
  \qquad
  \operatorname{Qcoh}(M^\dagger)
    \simeq\operatorname{Qcoh}^{H}(B^\vee)
    \simeq\operatorname{Qcoh}([B^\vee/H]),
\]
and similarly for coherent sheaves.

Apply the specialization of Polishchuk's theory just recalled.  Its
hypotheses hold: $\Spec R$ is affine and hence semi-separated, $B$ and
$B^\vee$ are proper and flat, the Poincar\'{e} kernel is a line bundle and its
adjoint inverse is a shift of a line bundle, and $E$ and $H$ are finite flat.
Thus the concentration and finite-Tor-dimension hypotheses hold.
Polishchuk's \cite[Theorem~2.5.1(i), (iii)]{PolishchukKernel} therefore gives
the two equivalences with $[B^\vee/H]$.  Finally,
\cref{thm:group-devissage} identifies this quotient with the actual
commutative group-stack dual $\Dual P$.
\end{proof}

\begin{remark}
The two extremal cases recover familiar equivalences.  If $E=0$, then
$P=B$ and \cref{eq:fourier-mukai} is classical Mukai duality for $B$ and
$B^\vee$.  If $B=0$, then $P=E$ and $\Dual P=\BG{E^{\mathrm D}}$; the theorem
is the derived categorical form of finite Cartier duality.
\end{remark}

The quasiabelian case puts the dual category on the smooth normalization.

\begin{corollary}\label{cor:fourier-mukai-quasiabelian}
Let $N$ be quasiabelian over $R$, choose an ample invertible sheaf
$\cL$ on $N$, let $u:A\to N$ be its normalization, and retain the
polarization data $K$, $F$, and $h:K\to F$ of \cref{thm:pushout}.  There is an
exact equivalence
\begin{equation}\label{eq:fourier-mukai-quasiabelian}
  \DbCoh{N}
  \simeq \DbCoh{[A/F^{\mathrm D}]}
  \simeq
  \mathbf D^{\mathrm b}
  \bigl(\operatorname{Coh}^{F^{\mathrm D}}(A)\bigr).
\end{equation}
Here $F^{\mathrm D}$ acts on $A$ by translation through
\[
  F^{\mathrm D}\xrightarrow{h^{\mathrm D}}K^{\mathrm D}
  \xrightarrow{\iota_{u^*\cL}}A.
\]
Thus the bounded coherent derived category of the possibly singular and
nonnormal model $N$ is encoded by its abelian normalization together with one
finite-flat translation action.
\end{corollary}

\begin{proof}
Apply \cref{thm:fourier-mukai} to the polarization quotient
$0\to F\to N\to A^\vee\to0$.  The quotient presentation and the displayed
action are \cref{thm:dual-quasiabelian}.
\end{proof}

More generally we have:

\begin{corollary}[Derived Mukai duality without tameness]
\label{cor:abelian-stack-fourier-mukai}
Let $\cA$ be an abelian stack over a discrete valuation ring $R$, in the sense
of Brochard, and put $G=\Dual{\cA}$.  Then
\[
  \DQcoh{\cA}\simeq\DQcoh{G},
  \qquad
  \DbCoh{\cA}\simeq\DbCoh{G}.
\]
No tameness assumption and no assumption on $2$ are required.
\end{corollary}

\begin{proof}
Brochard proves that $G$ is a duabelian group scheme
\cite[Theorem~3.17(2)]{BrochardDuality}, and that the abelian stack $\cA$ is
dualizable \cite[Example~4.10]{BrochardDuality}.  Hence the evaluation map
$\cA\to\Dual G$ is an isomorphism.  Apply \cref{thm:fourier-mukai} to $G$ and
reverse the displayed equivalences.
\end{proof}

For the derived categories of the abelian categories of sheaves,
\cref{cor:abelian-stack-fourier-mukai} gives an analogue without a tameness
assumption of
\cite{DhillonNasserdenMukai}.  That paper proves a stable
$\infty$-categorical statement in the tame case.

\subsection{Singularities.}

The structure theorem also extends \cref{cor:gorenstein} from
quasiabelian models to arbitrary proper group models.

\begin{corollary}\label{cor:general-gorenstein}
Let $R$ be a discrete valuation ring and let $P$ be a proper flat finitely
presented commutative $R$-group algebraic space.  Then $P\to\Spec R$ is
Gorenstein.  Its relative dualizing sheaf $\omega_{P/R}$ is translation
invariant and noncanonically trivial.  If $g=\dim P_\generic$, a choice of
trivialization gives a derived self-duality
\[
  \RHom_R\bigl(\mathbf R\Gamma(P,\cO_P),R\bigr)
  \simeq \mathbf R\Gamma(P,\cO_P)[g].
\]
\end{corollary}

\begin{proof}
Choose $0\to E\to P\xrightarrow{f}B\to0$ as in
\cref{thm:group-devissage}.  The map $f$ is an $E$-torsor.  As in the proof of
\cref{cor:gorenstein}, a finite locally free Hopf algebra over the local ring
$R$ is Frobenius, so $E/R$ is Gorenstein.  After the faithfully flat base
change $P\to B$, the morphism $f$ becomes the projection
$E\times_RP\to P$.  Hence $f$ is finite flat Gorenstein by descent.  Since
$B/R$ is smooth, the composite $P\to\Spec R$ is Gorenstein.

The universal-translation argument in the proof of
\cref{cor:gorenstein} applies verbatim to $P$ and gives
\[
  \omega_{P/R}\simeq p^*e^*\omega_{P/R},
  \qquad p:P\longrightarrow\Spec R,
\]
where $e$ is the identity section.  The invertible $R$-module
$e^*\omega_{P/R}$ is free because $R$ is local.  A generator therefore
trivializes $\omega_{P/R}$.

Put $g=\dim P_\generic$.  Since $f:P\to B$ is finite flat and $B/R$ is
smooth of relative dimension $g$, the structure morphism
$p:P\to\Spec R$ has pure relative dimension $g$.  Hence
\[
  p^!\cO_{\Spec R}\simeq\omega_{P/R}[g]
\]
\cite[Tags~0BV8 and~0C08]{Stacks}.  With the chosen trivialization,
Grothendieck duality gives
\[
  \RHom_R\bigl(\mathbf R\Gamma(P,\cO_P),R\bigr)
  \simeq \mathbf R\Gamma(P,p^!\cO_{\Spec R})
  \simeq \mathbf R\Gamma(P,\cO_P)[g]
\]
\cite[Corollary~(4.2.2)]{LipmanDuality}.
\end{proof}

\section{Arithmetic consequences}\label{sec:arithmetic}

The local theorem glues over a one-dimensional arithmetic base.

\begin{theorem}\label{thm:dedekind}
Let $S$ be a connected Dedekind scheme with generic point $\generic$, and
let $N$ be a proper flat finitely presented group algebraic space over $S$
whose generic fibre is an abelian variety.  Then $N$ is a commutative
projective group scheme.  The global \Neron{} model $A$ of $N_\generic$ is an
abelian scheme, and the generic identity extends to a finite normalization
map
\[
  u:A\longrightarrow N.
\]
For every $S$-ample invertible sheaf $\cL$ on $N$, the polarization
construction gives an exact sequence
\[
  0\longrightarrow F_\cL\longrightarrow N
  \xrightarrow{q_\cL}A^\vee\longrightarrow0
\]
with $F_\cL$ finite locally free.
\end{theorem}

\begin{proof}
Raynaud's representability theorem shows that $N$ is a scheme
\cite[Theorem~3.3.1]{RaynaudSpecialisation}.  Commutativity follows from
schematic density of the generic fibre, as in
\cref{lem:auto-commutative}.

The global \Neron{} model $A$ exists by
\cite[Chapter~1, \S4, Theorem~3]{BLR}.  For every closed point $s\in S$, its
base change to $\Spec\cO_{S,s}$ is the local \Neron{} model.  Applying
\cref{prop:normalization} there shows that this base change is an abelian
scheme.  Properness and geometric connectedness are local on $S$, so $A$ is
an abelian scheme.

The local theorem also gives a homomorphism
\[
  u_s:A_{\cO_{S,s}}\longrightarrow N_{\cO_{S,s}}
\]
extending the generic identity.  Since both schemes are finitely presented,
the limit theorem for morphisms descends $u_s$ to some open neighbourhood of
$s$ \cite[Tag~01ZB]{Stacks}.  Any two such extensions agree: the source is
flat, so its generic fibre is schematically dense, and the target is
separated.  The local maps
therefore glue with the generic identity to $u:A\to N$.

The map $u$ is proper because $A/S$ is proper and $N/S$ is separated.  Its
geometric fibres are finite by the local theorem, so it is quasi-finite and
hence finite.  It is surjective and birational, again locally on $S$.
Finally, the affine integral-closure argument in the proof of
\cref{prop:normalization} identifies $u$ with the normalization.

It remains to prove projectivity.  There is a dense open subscheme $U\subset S$
over which $u$ is an isomorphism.  Its complement is a finite set of closed
points $\{s_1,\ldots,s_r\}$.  The proof of \cref{prop:normalization} identifies
the base change of $u$ to $\Spec\cO_{S,s_i}$ with the group smoothening
(lissification) of $N_{\cO_{S,s_i}}$.  By
\cite[Chapter~II, Corollary~2.2.10 and its proof]{Anantharaman}, there exist an
integer $d_i>0$ and a homomorphism
\[
  v_i:N_{\cO_{S,s_i}}\longrightarrow A_{\cO_{S,s_i}}
  \qquad\text{such that}\qquad
  u\circ v_i=[d_i]_N.
\]
Put $d=\prod_i d_i$, with $d=1$ if $r=0$, and replace $v_i$ by
$[d/d_i]_A\circ v_i$.  By the limit theorem \cite[Tag~01ZB]{Stacks}, after shrinking to an open
neighbourhood $V_i$ of $s_i$, these maps and identities spread to homomorphisms
\[
  v_i:N_{V_i}\longrightarrow A_{V_i},
  \qquad u\circ v_i=[d]_N.
\]
Over $U$, put
$v_U=[d]_A\circ u_U^{-1}$.  All these maps restrict to $[d]$ on the generic fibre.
Since $N$ is flat over the integral scheme $S$, its generic fibre is
schematically dense over every nonempty overlap, and separatedness of $A$
shows that the maps agree.  They therefore glue to a homomorphism
\[
  v:N\longrightarrow A,
  \qquad u\circ v=[d]_N.
\]

The generic fibre $v_\generic=[d]_{A_\generic}$ has finite, hence affine,
kernel.  The homomorphism $v$ is separated because $N$ is separated over
$S$, so
\cite[Chapter~II, Proposition~2.3.2]{Anantharaman} shows that it is affine.  It
is also proper because $N$ is proper and $A$ is separated over $S$
\cite[Tag~01W6]{Stacks}; hence it is finite \cite[Tag~01WN]{Stacks}.  Finally,
the normal scheme $S$ is geometrically unibranch \cite[Tag~0BQ3]{Stacks}, so
the abelian scheme $A$ is projective by
\cite[Th\'{e}or\`{e}me~XI.1.4]{RaynaudAmple}.  Finite morphisms are projective,
and projective morphisms compose over the noetherian scheme $S$
\cite[Tags~0B3I and~0C4P]{Stacks}; thus $N$ is projective over $S$.

The last assertion is \cref{prop:conditional-polarization}.
\end{proof}

\begin{corollary}\label{cor:good-reduction}
An abelian variety over the function field of a connected Dedekind
scheme has good reduction everywhere if and only if it admits a proper flat
finitely presented group model.  In particular, over a discrete valuation
ring, the existence of such a model is equivalent to good reduction.
\end{corollary}

\begin{proof}
An abelian scheme is itself such a model.  Conversely, apply
\cref{thm:dedekind}, or \cref{prop:normalization} over a single discrete
valuation ring.  Thus a nonnormal proper group model cannot conceal bad
reduction.
\end{proof}

Finally we record a rigidity result that ensures all quasiabelian schemes 
over a DVR of small absolute ramification index are smooth.

\begin{theorem}\label{thm:low-ramification}
Let $R$ be a mixed-characteristic $(0,p)$ discrete valuation ring with
absolute ramification index $e_R=v_R(p)$.  If $e_R<p-1$, every quasiabelian
$R$-group scheme is an abelian scheme.
\end{theorem}

Here $v_R:\operatorname{Frac}(R)^\times\to\mathbf Z$ is the normalized
valuation, characterized by $v_R(\pi)=1$ for a uniformizer $\pi$ of $R$.
Thus $e_R=v_R(p)$ is the unique positive integer satisfying
$pR=\mathfrak m_R^{e_R}$.

\begin{proof}
Use the pushout diagram $h:K\to F$ of \cref{thm:pushout}.  Put
\[
  n=\operatorname{rk}_R(K)=\operatorname{rk}_R(F)
  \qquad\text{and write}\qquad
  n=p^a m,\quad (m,p)=1.
\]
The generic fibres of $K$ and $F$ are finite \'{e}tale commutative group
schemes of order $n$, and hence are annihilated by $n$.  Since $K$ and $F$
are flat over $R$, their generic fibres are schematically dense.  It follows
that $[n]_K=0=[n]_F$.

Choose $c,d\in\mathbf Z$ with
\[
  cm+dp^a=1.
\]
Modulo $n$, the classes
\[
  \epsilon_p=cm,\qquad \epsilon_{p'}=dp^a
\]
are the two complementary Chinese-remainder idempotents.  For either
$G=K$ or $G=F$, set
\[
  G_p:=G[p^a]=\Ker([p^a]_G),
  \qquad
  G_{p'}:=G[m]=\Ker([m]_G).
\]
Multiplication by the two idempotents gives an isomorphism
\[
  G\xrightarrow{\ \sim\ }G_p\times_R G_{p'},
  \qquad
  x\longmapsto\bigl([\epsilon_p]x,[\epsilon_{p'}]x\bigr),
\]
whose inverse is $(x,y)\mapsto x+y$.  Thus $G_p$ and $G_{p'}$ are finite
flat direct factors of $G$.  The first has $p$-power order, whereas the
second is killed by $m$ and is therefore finite \'{e}tale, since $m$ is
invertible in $R$.  These factors do not depend on the choice of $c$ and
$d$: the classes $\epsilon_p$ and $\epsilon_{p'}$ are uniquely determined
in $\mathbf Z/n\mathbf Z$.

The homomorphism $h$ commutes with multiplication by integers, so it respects
these decompositions and can be written
\[
  h=h_p\times h_{p'},
  \qquad
  h_p:K_p\longrightarrow F_p,
  \qquad
  h_{p'}:K_{p'}\longrightarrow F_{p'}.
\]
Both maps are isomorphisms on the generic fibre.  For the prime-to-$p$
factors, the generic inverse of $h_{p'}$ extends uniquely over $R$ by full
faithfulness for finite \'{e}tale schemes over a normal integral base.  For
the $p$-primary factors, the generic inverse of $h_p$ extends over $R$ by
finite-flat full faithfulness: under $e_R<p-1$, the exponent in
\cite[Example~1 and Corollary~2]{VasiuZink} is zero.

The two extensions combine to a homomorphism $g:F\to K$ extending the
generic inverse of $h$.  The composites $gh$ and $hg$ agree with the
identity on the schematically dense generic fibres of $K$ and $F$, hence
agree with the identity everywhere.  Thus $h$ is an isomorphism, and
\cref{cor:defect-rigidity} shows that $u:A\to N$ is an isomorphism.
\end{proof}

\begin{remark}
The ramification bound is sharp for the finite-flat uniqueness input.  For
$R=\mathbf Z_p[\zeta_p]$, one has $e_R=p-1$; the generic fibres of
$(\mathbf Z/p\mathbf Z)_R$ and $\mu_p$ are isomorphic, whereas their special
fibres are respectively \'{e}tale and connected.  Thus the generic inverse
need not extend at the boundary.
\end{remark}

There is a uniform statement even beyond the full-faithfulness range.

\begin{proposition}\label{prop:uniform-exponent}
Let $R$ be a mixed-characteristic $(0,p)$ discrete valuation ring.  There is
an integer $s=s(R)\geq0$ such that, for every quasiabelian model $N$ with
normalization $u:A\to N$, there is a finite faithfully flat homomorphism
$w:N\to A$ satisfying
\[
  wu=[p^s]_A,\qquad uw=[p^s]_N.
\]
In particular, the kernel and cokernel of every pushout map $h:K\to F$ are
annihilated by an exponent depending only on $R$.
\end{proposition}

\begin{proof}
The boundedness theorem of Vasiu and Zink supplies $s=s(R)$ and a
quasi-inverse on the $p$-primary factors of $h$
\cite[Theorem~1 and Corollary~2]{VasiuZink}.  On the prime-to-$p$ factors,
extend the generic inverse and compose it with $[p^s]$.  Reassembling gives
$h':F\to K$ such that
\[
  h'h=[p^s]_K,\qquad hh'=[p^s]_F.
\]
The homomorphism
\[
  A\times_R F\longrightarrow A,
  \qquad(a,f)\longmapsto[p^s]a+h'(f),
\]
is invariant under $k\mapsto(k,-h(k))$ and descends through
\cref{thm:pushout} to $w:N\to A$.  The two composite identities follow
immediately.

The map $w$ is proper, and $wu=[p^s]_A$ is finite and surjective.  Since $u$
is finite and surjective, this forces $w$ to be surjective with finite
fibres, hence finite.  Its source is Cohen--Macaulay, its target is regular,
and the local dimensions agree; miracle flatness makes it finite faithfully
flat.
\end{proof}

\section{Examples and final remarks}\label{sec:examples}

\begin{example}
If $N=A$ is already an abelian scheme, then the normalization map is the
identity and $q_\cL=\phi_\cL$.  The structure theorem reduces
to the classical polarization isogeny.  In
\eqref{eq:dual-quasiabelian-quotient}, the action of $K^{\mathrm D}$ is the
dual-polarization kernel, and the quotient is the classical dual abelian
scheme $A^\vee$.
\end{example}

\begin{example}\label{ex:serre}
Serre constructed, over a ramified discrete valuation ring of mixed
characteristic, a homomorphism of abelian schemes whose schematic image $N$ is
flat and proper but is not an abelian scheme; see
\cite[\S7.5, Example~8]{BLR} and
\cite[Example~4.11]{AchterCasalainaMartinWise}.  Fakhruddin and Srinivas give
an equal-characteristic version
\cite[\S4, item~2]{FakhruddinSrinivas}.  These examples show that the stronger
claim ``every quasiabelian model is abelian'' is false.

Nevertheless, \cref{thm:polarization-quotient} supplies a finite faithfully
flat map $N\To\AbDual A$.  The normalization $A\To N$ generally has nonflat
kernel, while the polarization quotient has a different, finite flat kernel.
This distinction is exactly what permits duality to remain proper.
The defect $e(N)$ is greater than one; by
\cref{cor:thickness-stackiness}, the dual has trivial generic stabilizer and
a nontrivial, vertically supported special stabilizer.  Thus its inertia is
finite but not flat.
\end{example}

\begin{remark}
For an $R$-scheme $T$, an object of the dual
\[
  \Dual N(T)=\HomStack(N_T,\BG{\Gm})
\]
is a multiplicative invertible sheaf on $N_T$: an invertible sheaf $\cP$
equipped with a coherent isomorphism
\[
  m^*\cP\simeq p_1^*\cP\otimes p_2^*\cP.
\]
An ordinary character $\chi:N_T\to\Gm$ is different: it is a multiplicative
automorphism of the neutral object of $\Dual N(T)$.  Thus
$\operatorname{Hom}_R(N,\Gm)$ records only the $R$-points of the stabilizer
sheaf of the neutral object; it does not describe all objects and
isomorphisms in $\Dual N$ after base change.  In particular, its vanishing
does not control the relative sheaves
$\Isom^{\mathrm{mult}}(\cP,\cQ)$, which are the pullbacks of the diagonal of
$\Dual N$, and therefore does not prove that this diagonal is proper.

The quotient presentation $\Dual N\simeq[A/F^{\mathrm D}]$, represented by
the two-term complex \eqref{eq:dual-quasiabelian}, retains this information.
Fppf-locally, objects lift to points $a,a'$ of $A$, and a morphism from $a$ to
$a'$ is an element $f\in F^{\mathrm D}$ such that
$a'=a+\delta(f)$.  Since $F^{\mathrm D}$ is finite and $A$ is separated, the
sheaf of such elements is finite.  Equivalently, the diagonal of $\Dual N$ is
finite, hence proper.  This gives the required separation statement without
replacing the relative character sheaf by its global sections.
\end{remark}

\begin{remark}
The hypothesis that $2$ be invertible is needed only when we pass from group
algebraic spaces to group stacks with nontrivial inertia.  For a proper flat
finitely presented commutative group algebraic space $P$, the conclusions
\[
  \Dual P\text{ is proper and flat with finite diagonal},
  \qquad P\xrightarrow{\sim}\DoubleDual P,
\]
hold in every characteristic by \cref{thm:global-group-duality}.  Over a
discrete valuation ring, we have a 
presentation $0\to E\to P\to B\to0$ and the quotient formula
\[
  \Dual P\simeq[B^\vee/E^{\mathrm D}].
\]
This formula uses only $\ExtSheaf{1}{E}=0$.  Biduality additionally uses
fibrewise dualizability and the globalization argument in the proof of
\cref{thm:global-group-duality}; it does not require the vanishing of
$\ExtSheaf{2}{P}$.  Indeed, the truncation in
\eqref{eq:dual-truncation} does not retain this higher Ext sheaf, so the 
geometry of $\Dual P$ and biduality do not imply that
$\ExtSheaf{2}{P}=0$.

The situation changes for a commutative group stack $\cG$.  Put
\[
  H=\HHm(\cG),\qquad P=\HHz(\cG).
\]
The canonical sequence
\[
  0\longrightarrow\BG H\longrightarrow\cG\longrightarrow P
  \longrightarrow0
\]
exhibits $\cG\to P$ as an $H$-gerbe.  Restriction to the inertia gives a map
$\Dual{\cG}\to H^{\mathrm D}$.  Let $T\to S$ and
$\chi\in H^{\mathrm D}(T)$.  A global lift of $\chi$ to
$\Dual{\cG}(T)$ is a multiplicative invertible sheaf on $\cG_T$ whose
restriction to the inertia is $\chi$.  Pushing out the $H_T$-gerbe
$\cG_T\to P_T$ along $\chi$ gives a $\Gm$-gerbe with a Yoneda class
\[
  \partial_T(\chi)\in
  \operatorname{Ext}^{2}_{T,\fppf}(P_T,\Gm).
\]
The character $\chi$ has a global lift if and only if this class is zero.
After fppf sheafification, the classes $\partial_T(\chi)$ form the
homomorphism
\[
  o_{\cG}:H^{\mathrm D}\longrightarrow\ExtSheaf{2}{P},
\]
and $o_{\cG}(\chi)=0$ means that $\partial_T(\chi)$ becomes zero after an
fppf cover of $T$.  Equivalently, $\chi$ lifts fppf-locally to
$\Dual{\cG}$.  This local lifting condition, rather than surjectivity on
global $T$-objects, is what exactness of group stacks requires.
Indeed, the relevant part of the cohomology sequence is
\[
  0\longrightarrow\ExtSheaf{1}{P}
  \longrightarrow\HHz(\Dual{\cG})
  \longrightarrow H^{\mathrm D}
  \xrightarrow{o_{\cG}}\ExtSheaf{2}{P}.
\]
Thus the expected dual sequence
\[
  0\longrightarrow\Dual P\longrightarrow\Dual{\cG}
  \longrightarrow H^{\mathrm D}\longrightarrow0
\]
is exact if and only if $o_{\cG}$ is zero.

The distinction between global and local lifting matters here.  If
$P=A$ is an abelian scheme, then
\cref{eq:global-ext2-abelian} gives
\[
  \operatorname{Ext}^{2}_{T,\fppf}(A_T,\Gm)
  \simeq H^1_{\fppf}(T,A_T^\vee),
\]
which can be nonzero, whereas $\ExtSheaf{2}{A}=0$.  Sheaf-level exactness
therefore does not assert that every character over $T$ has a global lift.
If $\cG=P$ is a group algebraic space, then $H=0$, so the obstruction has
zero source even when $\ExtSheaf{2}{P}$ is nonzero.

For the finite-by-abelian presentation
\[
  0\longrightarrow E\longrightarrow P\longrightarrow B\longrightarrow0,
\]
the long Ext sequence and the characteristic-free vanishing
$\ExtSheaf{2}{B}=0$ give a monomorphism
\[
  \ExtSheaf{2}{P}\longrightarrow\ExtSheaf{2}{E}.
\]
Consequently, the abelian quotient contributes no second-Ext obstruction;
the remaining possible contribution is controlled by the finite flat kernel.
When $2$ is invertible, the finite-group vanishing in
\cref{prop:ext-vanishing} kills this term and therefore $o_{\cG}$.
In characteristic $2$ the obstruction can in fact be nonzero, and biduality
can fail.  Let $k$ be an algebraically closed field of characteristic $2$.
Breen constructs a nonzero class
\[
  \varepsilon\in
  \operatorname{Ext}^{2}_{k,\fppf}(\alpha_2,\Gm)
\]
\cite{BreenNontrivial}.  This class remains nonzero in
$\ExtSheaf{2}{\alpha_2}(k)$.  Indeed, the local-to-global spectral sequence,
together with
\[
  \ExtSheaf{1}{\alpha_2}=0,\qquad
  \mathop{\underline{\operatorname{Hom}}}\nolimits(\alpha_2,\Gm)
    \simeq\alpha_2,
\]
identifies the global and internal groups in degree $2$, since
$H^i_{\fppf}(k,\alpha_2)=0$ for $i>0$.  The latter vanishing follows from
\[
  0\longrightarrow\alpha_2\longrightarrow\mathbf G_{\mathrm a}
  \xrightarrow{F}\mathbf G_{\mathrm a}\longrightarrow0,
\]
the perfection of $k$, and the fppf acyclicity of $\mathbf G_{\mathrm a}$.

Since $[2]_{\alpha_2}=0$, multiplication by $2$ kills $\varepsilon$.
The Kummer sequence
\[
  0\longrightarrow\mu_2\longrightarrow\Gm
  \xrightarrow{[2]}\Gm\longrightarrow0
\]
therefore lifts $\varepsilon$ to
$\kappa\in\operatorname{Ext}^{2}_{k,\fppf}(\alpha_2,\mu_2)$.
Under Deligne's equivalence, $\kappa$ determines an exact sequence
\[
  0\longrightarrow\BG{\mu_2}\longrightarrow\cG_\kappa
  \longrightarrow\alpha_2\longrightarrow0.
\]
Thus $\cG_\kappa\to\alpha_2$ is a $\mu_2$-gerbe, so $\cG_\kappa$ is
algebraic, proper, flat, and finitely presented, with finite flat inertia.
For the nontrivial character $\chi:\mu_2\to\Gm$, pushout gives
\[
  o_{\cG_\kappa}(\chi)=\varepsilon\ne0.
\]
Using \eqref{eq:hminus-dual} and \eqref{eq:hzero-dual}, we obtain
\[
  \HHm(\Dual{\cG_\kappa})\simeq\alpha_2,\qquad
  \HHz(\Dual{\cG_\kappa})=
  \Ker\!\left((\mathbf Z/2\mathbf Z)_k
    \xrightarrow{\,1\mapsto\varepsilon\,}
    \ExtSheaf{2}{\alpha_2}\right).
\]
The kernel in the second formula is an algebraic subgroup of the finite
\'{e}tale constant group $(\mathbf Z/2\mathbf Z)_k$.  Over the algebraically
closed field $k$, any nontrivial such subgroup contains the nonidentity
$k$-point, but that point maps to $\varepsilon\ne0$.  Hence the kernel is
trivial.
Hence $\Dual{\cG_\kappa}\simeq\BG{\alpha_2}$ and
$\DoubleDual{\cG_\kappa}\simeq\alpha_2$.  The evaluation morphism kills
the nontrivial $\mu_2$-inertia and is therefore not an isomorphism.
Thus biduality genuinely fails in characteristic $2$, and the hypothesis
that $2$ be invertible is essential for the group-stack theorem.
\end{remark}

\begin{remark}
The analysis of Achter, Casalaina-Martin, and Wise identifies precise
conditions under which the schematic image of a homomorphism of abelian
schemes is again an abelian scheme
\cite[Theorem~A]{AchterCasalainaMartinWise}.  Their
\cite[Examples~4.11 and~4.12]{AchterCasalainaMartinWise} show that flatness
of the image is not enough.  The present theorem on finite extensions of
abelian schemes addresses a different question: it retains a useful
polarization quotient even when the image is nonnormal and hence not abelian.
\end{remark}

\section*{Statement on the use of artificial intelligence}

The authors used OpenAI's ChatGPT/Codex  to assist in developing and checking
mathematical arguments, revising the exposition, and identifying typesetting
and consistency errors.  All AI-assisted suggestions were independently
checked by the authors, who take full responsibility for the entire content of
the manuscript.


\begin{thebibliography}{AGV72--73}

\bibitem[ACMW24]{AchterCasalainaMartinWise}
J.~D. Achter, S. Casalaina-Martin, and J. Wise,
\emph{Images of abelian schemes},
preprint, \href{https://arxiv.org/abs/2312.13262v2}{arXiv:2312.13262v2},
2024.

\bibitem[Ana73]{Anantharaman}
S. Anantharaman,
\emph{Sch\'{e}mas en groupes, espaces homog\`{e}nes et espaces alg\'{e}briques
sur une base de dimension $1$},
M\'{e}m. Soc. Math. France \textbf{33} (1973), 5--79.

\bibitem[AGV72--73]{SGA4}
M. Artin, A. Grothendieck, and J.-L. Verdier (eds.),
\emph{Th\'{e}orie des topos et cohomologie \'{e}tale des sch\'{e}mas
(SGA~4)},
Lecture Notes in Math., vols.~269, 270, 305, Springer-Verlag, 1972--1973.

\bibitem[Ber09]{BertolinExtensions}
C. Bertolin,
\emph{Extensions and biextensions of locally constant group schemes, tori
and abelian schemes},
Math. Z. \textbf{261} (2009), no.~4, 845--868.

\bibitem[BLR90]{BLR}
S. Bosch, W. L\"{u}tkebohmert, and M. Raynaud,
\emph{N\'{e}ron Models},
Ergeb. Math. Grenzgeb. (3), vol.~21, Springer-Verlag, Berlin, 1990.

\bibitem[Bre69]{BreenNontrivial}
L. Breen,
\emph{On a nontrivial higher extension of representable abelian sheaves},
Bull. Amer. Math. Soc. \textbf{75} (1969), 1249--1253,
\href{https://doi.org/10.1090/S0002-9904-1969-12379-7}
{doi:10.1090/S0002-9904-1969-12379-7}.

\bibitem[Bre70]{BreenExtensions}
L. Breen,
\emph{Extensions of abelian sheaves and Eilenberg--Mac Lane algebras},
Invent. Math.~\textbf{9} (1969/70), 15--44.

\bibitem[Bre75]{BreenVanishing}
L. Breen,
\emph{Un th\'{e}or\`{e}me d'annulation pour certains
$\operatorname{Ext}^{i}$ de faisceaux ab\'{e}liens},
Ann. Sci. \'{E}cole Norm. Sup. (4) \textbf{8} (1975), 339--352.

\bibitem[Bro21]{BrochardDuality}
S. Brochard,
\emph{Duality for commutative group stacks},
Int. Math. Res. Not. IMRN \textbf{2021} (2021), no.~3, 2321--2388.

\bibitem[CZ17]{ChenZhu}
T.-H. Chen and X. Zhu,
\emph{Geometric Langlands in prime characteristic},
Compos. Math. \textbf{153} (2017), no.~2, 395--452,
\href{https://doi.org/10.1112/S0010437X16008113}
{doi:10.1112/S0010437X16008113}.

\bibitem[DG11]{SGA3I}
M. Demazure and A. Grothendieck (eds.),
\emph{Sch\'{e}mas en groupes (SGA~3), Tome~I: Propri\'{e}t\'{e}s
g\'{e}n\'{e}rales des sch\'{e}mas en groupes},
Documents Math\'{e}matiques, vol.~7, Soc. Math. France, Paris, 2011.

\bibitem[DN23]{DhillonNasserdenMukai}
A. Dhillon and B. Nasserden,
\emph{Mukai duality for abelian stacks},
\href{https://arxiv.org/abs/2311.11492}{arXiv:2311.11492}, 2023.

\bibitem[DP08]{DonagiPantev}
R. Donagi and T. Pantev,
\emph{Torus fibrations, gerbes, and duality},
Mem. Amer. Math. Soc. \textbf{193} (2008), no.~901, vi+90 pp.,
with an appendix by D. Arinkin,
\href{https://doi.org/10.1090/memo/0901}{doi:10.1090/memo/0901}.

\bibitem[FS08]{FakhruddinSrinivas}
N. Fakhruddin and V. Srinivas,
\emph{A topological property of quasireductive group schemes},
Algebra Number Theory \textbf{2} (2008), no.~2, 121--134.

\bibitem[SGA1]{SGA1}
A. Grothendieck,
\emph{Rev\^{e}tements \'{e}tales et groupe fondamental (SGA~1)},
Documents Math\'{e}matiques, vol.~3, Soc. Math. France, Paris, 2003,
recomposed and annotated by M. Raynaud.

\bibitem[SGA7 I]{SGA7}
A. Grothendieck,
\emph{Groupes de monodromie en g\'{e}om\'{e}trie alg\'{e}brique. I
(SGA~7~I)},
Lecture Notes in Math., vol.~288, Springer-Verlag, Berlin--New York, 1972,
with the collaboration of M. Raynaud and D.~S. Rim.

\bibitem[Lau96]{LaumonFourier}
G. Laumon,
\emph{Transformation de Fourier g\'{e}n\'{e}ralis\'{e}e},
preprint, \href{https://arxiv.org/abs/alg-geom/9603004}
{arXiv:alg-geom/9603004}, 1996.

\bibitem[Lip09]{LipmanDuality}
J. Lipman,
\emph{Notes on derived functors and Grothendieck duality},
in \emph{Foundations of Grothendieck Duality for Diagrams of Schemes},
Lecture Notes in Math., vol.~1960, Springer-Verlag, Berlin, 2009, 1--259.

\bibitem[Muk81]{MukaiDuality}
S. Mukai,
\emph{Duality between $D(X)$ and $D(\widehat X)$ with its application to
Picard sheaves},
Nagoya Math. J. \textbf{81} (1981), 153--175,
\href{https://doi.org/10.1017/S002776300001922X}
{doi:10.1017/S002776300001922X}.

\bibitem[Mum74]{Mumford}
D. Mumford,
\emph{Abelian Varieties},
2nd ed., Oxford University Press, London, 1974.

\bibitem[Oor66]{Oort}
F. Oort,
\emph{Commutative Group Schemes},
Lecture Notes in Math., vol.~15, Springer-Verlag, Berlin--New York, 1966.

\bibitem[Par71]{Pareigis}
B. Pareigis,
\emph{When Hopf algebras are Frobenius algebras},
J. Algebra \textbf{18} (1971), no.~4, 588--596.

\bibitem[Pol11]{PolishchukKernel}
A. Polishchuk,
\emph{Kernel algebras and generalized Fourier--Mukai transforms},
J. Noncommut. Geom. \textbf{5} (2011), no.~2, 153--251,
\href{https://doi.org/10.4171/JNCG/73}{doi:10.4171/JNCG/73}.

\bibitem[PR11]{PoonenRainsSelfCup}
B. Poonen and E. Rains,
\emph{Self cup products and the theta characteristic torsor},
Math. Res. Lett.~\textbf{18} (2011), no.~6, 1305--1318.

\bibitem[Ray70a]{RaynaudAmple}
M. Raynaud,
\emph{Faisceaux amples sur les sch\'{e}mas en groupes et les espaces
homog\`{e}nes},
Lecture Notes in Math., vol.~119, Springer-Verlag, Berlin--New York, 1970.

\bibitem[Ray70b]{RaynaudSpecialisation}
M. Raynaud,
\emph{Sp\'{e}cialisation du foncteur de Picard},
Publ. Math. Inst. Hautes \'{E}tudes Sci.~\textbf{38} (1970), 27--76.

\bibitem[RR26]{RibeiroRosengarten}
G. Ribeiro and Z. Rosengarten,
\emph{Extensions of abelian schemes and the additive group},
Algebraic Geom., to appear,
\href{https://arxiv.org/abs/2506.17393v2}{arXiv:2506.17393v2}, 2026.

\bibitem[Rom12]{RomagnyEffective}
M. Romagny,
\emph{Effective models of group schemes},
J. Algebraic Geom.~\textbf{21} (2012), no.~4, 643--682.

\bibitem[Ryd15]{RydhApproximation}
D. Rydh,
\emph{Noetherian approximation of algebraic spaces and stacks},
J. Algebra \textbf{422} (2015), 105--147.

\bibitem[SZ95]{SilverbergZarhin}
A. Silverberg and Yu.~G. Zarhin,
\emph{Semistable reduction and torsion subgroups of abelian varieties},
Ann. Inst. Fourier (Grenoble) \textbf{45} (1995), no.~2, 403--420.

\bibitem[Sta]{Stacks}
The Stacks Project Authors,
\emph{Stacks Project},
\href{https://stacks.math.columbia.edu}{stacks.math.columbia.edu}.

\bibitem[TO70]{TateOort}
J.~Tate and F.~Oort,
\emph{Group schemes of prime order},
Ann. Sci. \'{E}cole Norm. Sup. (4) \textbf{3} (1970), no.~1, 1--21.

\bibitem[VZ12]{VasiuZink}
A.~Vasiu and T.~Zink,
\emph{Boundedness results for finite flat group schemes over discrete
valuation rings of mixed characteristic},
J. Number Theory \textbf{132} (2012), no.~9, 2003--2019.

\bibitem[Wat79]{Waterhouse}
W.~C. Waterhouse,
\emph{Introduction to Affine Group Schemes},
Graduate Texts in Mathematics, vol.~66,
Springer-Verlag, New York--Berlin, 1979.

\end{thebibliography}
\end{document}